\documentclass[12pt]{amsart}
\usepackage{amsmath,amsthm,amsfonts,latexsym}
\usepackage{dynkin-diagrams,gauss}
\usepackage{hyperref}
\usepackage{graphicx,tikz}
\usepackage{color}
\usepackage[shortlabels]{enumitem}
\setlist[enumerate,1]{label={\upshape(\roman*)}}

\usetikzlibrary{arrows,decorations.markings,shapes}
\tikzstyle{fat}=[circle, draw=black, fill=black, inner sep=0pt, minimum width=10pt]
\tikzstyle{vertex}=[circle, draw, fill=black, inner sep=0pt, minimum width=6pt]
\tikzstyle{vert}=[circle, draw=black, fill=white, inner sep=0pt, minimum width=6pt]
\tikzstyle{pc}=[circle, draw=black, inner sep=1pt, minimum width=10pt, font=\tiny] %positive charge
\tikzstyle{nc}=[circle, draw=black, inner sep=1pt, minimum width=10pt, font=\tiny] % style=dashed,
\tikzstyle{pedge}=[draw,-]
\tikzstyle{wedge}=[draw,-,postaction={decorate}, decoration={markings,mark = at position
0.7 with {\arrow{stealth} } }]
\tikzstyle{wmedge}=[draw,-,postaction={decorate}, decoration={markings,mark = at position
0.55 with {\arrow{stealth} } }]
\tikzstyle{dpedge}=[draw,very thick]
\tikzstyle{dwedge}=[draw,very thick,postaction={decorate}, decoration={markings,mark = at position
0.7 with {\arrow{stealth} } }]
\tikzstyle{wwedge}=[draw,-,postaction={decorate}, decoration={markings,mark = at position 0.5 with {\arrow{stealth} }, mark = at position 0.6 with {\arrow{stealth} } }]
\tikzstyle{wwedge2}=[draw,-,postaction={decorate}, decoration={markings,mark = at position 0.45 with {\arrow{stealth} }, mark = at position 0.65 with {\arrow{stealth} } }]
\tikzstyle{wwwedge}=[draw,-,postaction={decorate}, decoration={markings,mark = at position 0.65 with {\arrow{stealth} },mark = at position 0.45 with {\arrow{stealth} }, mark = at position 0.55 with {\arrow{stealth} } }]
\tikzstyle{wwwedge2}=[draw,-,postaction={decorate}, decoration={markings,mark = at position 0.7 with {\arrow{stealth} },mark = at position 0.4 with {\arrow{stealth} }, mark = at position 0.55 with {\arrow{stealth} } }]
\tikzstyle{wwwwedge}=[draw,-,postaction={decorate}, decoration={markings,mark = at position 0.7 with {\arrow{stealth} },mark = at position 0.4 with {\arrow{stealth} }, mark = at position 0.5 with {\arrow{stealth} }, mark = at position 0.6 with {\arrow{stealth} } }]
\tikzstyle{wwwnedge}=[draw,densely dashed,postaction={decorate}, decoration={markings,mark = at position 0.65 with {\arrow{stealth} },mark = at position 0.45 with {\arrow{stealth} }, mark = at position 0.55 with {\arrow{stealth} } }]
\tikzstyle{wwnedge}=[draw,densely dashed,postaction={decorate}, decoration={markings,mark = at position 0.5 with {\arrow{stealth} }, mark = at position 0.6 with {\arrow{stealth} } }]
\tikzstyle{wnedge}=[draw,densely dashed,postaction={decorate}, decoration={markings,mark = at position 0.55 with {\arrow{stealth} } }]
\tikzstyle{wnedge2}=[draw,densely dashed,postaction={decorate}, decoration={markings,mark = at position 0.6 with {\arrow{stealth} } }]
\tikzstyle{dnedge}=[draw,densely dashed,postaction={decorate}]
\tikzstyle{nedge}=[draw,densely dashed]
\tikzstyle{weight2}= [draw=white, fill=white, font=\scriptsize]
\tikzstyle{weight}= [font=\scriptsize]
\tikzstyle{empty}=[circle, draw=white, inner sep=2pt, fill=white, minimum width=4pt]
\tikzstyle{ghost}=[circle, draw=black, inner sep=1pt, style=densely dashed, minimum width=6pt, font=\tiny]
\tikzstyle{ghostc}=[circle, draw=black, inner sep=1pt, style=densely dashed, minimum width=10pt, font=\tiny]
\tikzstyle{dedge}=[draw,very thick,dotted]

\newtheorem{theorem}{Theorem}
\newtheorem{lemma}[theorem]{Lemma}
\newtheorem{proposition}[theorem]{Proposition}
\theoremstyle{definition}
\newtheorem{dfn}[theorem]{Definition}

\newcommand{\Z}{\mathbb{Z}}
\newcommand{\C}{\mathbb{C}}

\newcommand{\cH}{\mathcal{H}}
\newcommand{\cS}{\mathcal{S}}

\title[Hermitian spectral radius]{Maximal digraphs whose Hermitian spectral radius is at most $2$}
\author[A. L. Gavrilyuk]{Alexander L. Gavrilyuk}
\address{Interdisciplinary Faculty of Science and Engineering\\ Shimane University}
\email{gavrilyuk@riko.shimane-u.ac.jp}
\author[A. Munemasa]{Akihiro Munemasa}
\address{Graduate School of Information Sciences\\
Tohoku University}
\email{munemasa@math.is.tohoku.ac.jp}

\date{\today}
\begin{document}

\begin{abstract}
We classify maximal digraphs whose Hermitian spectral radius is at most $2$.
\end{abstract}
\maketitle

\section{Introduction}

Smith \cite{Smith} and Lemmens and Seidel \cite{LS} showed that a connected simple graph
whose $(0,1)$-adjacency matrix has spectral radius at most 2
is a subgraph of one of the following graphs:

\begin{figure}[h]
\begin{center}
\begin{tikzpicture}[auto, scale=0.7] % Dynkin A-extended
 \begin{scope}
  \foreach \type/\pos/\name in {
   {vertex/(0,-0.5)/a0}, {vertex/(1,-0.5)/a1}, {empty/(2,-0.5)/a2}, {empty/(3,-0.5)/a3},
   {vertex/(4,-0.5)/a4}, {vertex/(5,-0.5)/a5}, {vertex/(2.5,1)/b}}
   \node[\type] (\name) at \pos {};
  \foreach \edgetype\source/\dest in {
   {pedge/a0/a1}, {pedge/a1/a2}, {dedge/a2/a3}, {pedge/a3/a4}, {pedge/a4/a5},
   {pedge/a0/b}, {pedge/a5/b}}
  \path[\edgetype] (\source) -- (\dest);
 \end{scope}
\end{tikzpicture}
\hspace{1cm}
\begin{tikzpicture}[auto, scale=0.7] % Dynkin D-extended
 \begin{scope}
  \foreach \type/\pos/\name in {
   {vertex/(0,0)/a0}, {vertex/(1,0)/a1}, {empty/(2,0)/a2}, {empty/(3,0)/a3},
   {vertex/(4,0)/a4}, {vertex/(5,0)/a5},
   {vertex/(-0.5,0.7)/b1},{vertex/(-0.5,-0.7)/b0},
   {vertex/(5.5,0.7)/c1},{vertex/(5.5,-0.7)/c0}}
   \node[\type] (\name) at \pos {};
  \foreach \edgetype\source/\dest in {
   {pedge/a0/a1}, {pedge/a1/a2}, {dedge/a2/a3}, {pedge/a3/a4}, {pedge/a4/a5},
   {pedge/a0/b1}, {pedge/a0/b0}, {pedge/a5/c1}, {pedge/a5/c0}}
  \path[\edgetype] (\source) -- (\dest);
 \end{scope}
\end{tikzpicture}
\end{center}
\begin{center}
\begin{tikzpicture}[auto, scale=0.7] % Dynkin E6-extended
 \begin{scope}
  \foreach \pos/\name in {
   {(0,0)/a0}, {(1,0)/a1}, {(2,0)/a2}, {(3,0)/a3}, {(4,0)/a4},
   {(2,1)/b1}, {(2,2)/b2}}
   \node[vertex] (\name) at \pos {};
  \foreach \source/\dest in {
   {a0/a1}, {a1/a2}, {a2/a3}, {a3/a4}, {a2/b1}, {b1/b2}}
  \path[pedge] (\source) -- (\dest);
 \end{scope}
\end{tikzpicture}
\hspace{0.3cm}
\begin{tikzpicture}[auto, scale=0.7] % Dynkin E7-extended
 \begin{scope}
  \foreach \pos/\name in {
   {(0,0)/a0}, {(1,0)/a1}, {(2,0)/a2}, {(3,0)/a3}, {(4,0)/a4}, {(5,0)/a5}, {(6,0)/a6},
   {(3,1)/b1}}
   \node[vertex] (\name) at \pos {};
  \foreach \source/\dest in {
   {a0/a1}, {a1/a2}, {a2/a3}, {a3/a4}, {a4/a5}, {a5/a6}, {a3/b1}}
  \path[pedge] (\source) -- (\dest);
 \end{scope}
\end{tikzpicture}
\hspace{0.3cm}
\begin{tikzpicture}[auto, scale=0.7] % Dynkin E8-extended
 \begin{scope}
  \foreach \pos/\name in {
   {(0,0)/a0}, {(1,0)/a1}, {(2,0)/a2}, {(3,0)/a3}, {(4,0)/a4}, {(5,0)/a5}, {(6,0)/a6}, {(7,0)/a7},
   {(2,1)/b1}}
   \node[vertex] (\name) at \pos {};
  \foreach \source/\dest in {
   {a0/a1}, {a1/a2}, {a2/a3}, {a3/a4}, {a4/a5}, {a5/a6}, {a6/a7}, {a2/b1}}
  \path[pedge] (\source) -- (\dest);
 \end{scope}
\end{tikzpicture}
\end{center}
\caption{Extended Dynkin diagrams of types $\mathsf{A}$, $\mathsf{D}$, $\mathsf{E}$}
\label{fig:Dynkin}
\end{figure}
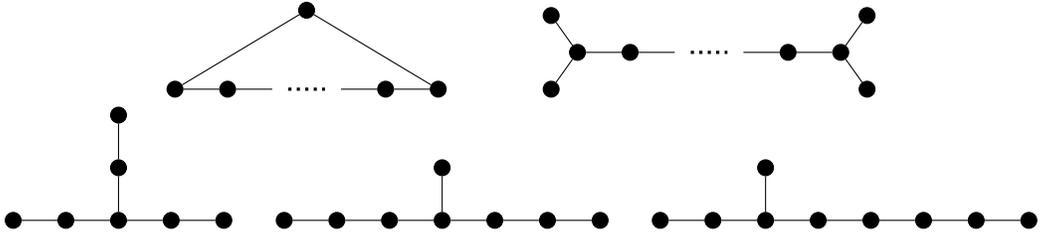

As is well known, these are extended Dynkin diagrams of the irreducible root lattices of types $\mathsf{A}$,
$\mathsf{D}$, and $\mathsf{E}$.
Indeed, if $A$ is the adjacency matrix of a graph with
spectral radius at most $2$, then
$-A+2I$
%$A+2I$ (and, similarly, $-A+2I$)
is a positive semidefinite matrix,
which thus can be seen as the Gram matrix of a set $\Sigma$ of vectors in $\mathbb{R}^n$
such that $({\bf x},{\bf x})=2$ and $({\bf x},{\bf y})\in \{0,-1\}$ for all ${\bf x}\ne {\bf y}\in \Sigma$.
Therefore, $\Sigma$ is contained in a fundamental root system of a root lattice \cite{CGSS}.

The aim of the present work is to generalize this result to the class of digraphs.
A {\bf digraph} (a {\bf directed} or {\bf mixed} graph) $\Delta$ consists
of a finite set $V$ of vertices together with a subset $E \subseteq  V \times V$
of ordered pairs of distinct elements of $V$.
If $(x,y)\in E$ and $(y,x)\notin E$, then
we call $(x,y)$ an {\bf arc} or {\bf directed edge}.
If both $(x,y)\in E$ and $(y,x)\in E$, then the pair
$\{x,y\}$ is said to form a {\bf digon} of $\Delta$.

The {\bf Hermitian adjacency matrix} $H=H(\Delta)$ of $\Delta$ was independently
defined by
Liu and Li \cite{LL}, and Guo and Mohar \cite{GM}
as a Hermitian matrix $H\in \mathbb{C}^{V\times V}$
with entries given by
\begin{align*}
(H)_{xy}=\begin{cases}
1 & \text{ if }(x,y)\in E,(y,x)\in E,\\
i & \text{ if } (x,y)\in E,(y,x)\notin E,\\
-i & \text{ if }(x,y)\notin E,(y,x)\in E,\\
0 & \text{ otherwise}.
\end{cases}
\end{align*}

In what follows, by the Hermitian
spectral radius, or simply, spectral radius
$\rho(\Delta)$ of a digraph
$\Delta$, we mean the spectral radius of
its Hermitian adjacency matrix $H(\Delta)$.
Guo and Mohar \cite{GM2} studied digraphs
whose spectral radius is less than $2$.
In this paper, we give a classification of digraphs
whose spectral radius is at most $2$.

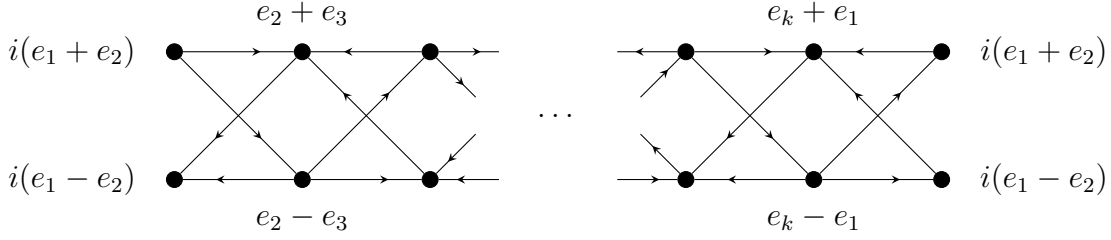
\begin{figure}
\begin{center}
\begin{tikzpicture}[auto, scale=1.7] % \Delta^{(1)}_{2k}$, $k$ even
 \begin{scope}
  \foreach \type/\pos/\name in {
   {vertex/(0,0)/a2}, {vertex/(0,1)/a1}, {vertex/(1,1)/b1}, {vertex/(1,0)/b2},
   {vertex/(2,0)/e2}, {vertex/(2,1)/e1}, {empty/(2.6,1)/b11}, {empty/(2.6,0)/b21},
   {empty/(2.4,0.6)/b12}, {empty/(2.4,0.4)/b22}, {empty/(3.4,1)/c11}, {empty/(3.4,0)/c21},
   {empty/(3.6,0.6)/c12}, {empty/(3.6,0.4)/c22}, {vertex/(4,1)/c1}, {vertex/(4,0)/c2},
   {vertex/(5,1)/d1}, {vertex/(5,0)/d2}, {vertex/(6,1)/f1}, {vertex/(6,0)/f2}}
   \node[\type] (\name) at \pos {};
  \foreach \pos/\name in {
   {(3,0.5)/\dots}, {(-0.8,1)/i(e_1+e_2)}, {(-0.8,0)/i(e_1-e_2)},
   {(1,1.3)/e_2+e_3}, {(1,-0.3)/e_2-e_3},
   {(5,1.3)/e_k+e_1}, {(5,-0.3)/e_k-e_1},
   {(6.8,1)/i(e_1+e_2)}, {(6.8,0)/i(e_1-e_2)}}
   \node at \pos {$\name$};
  \foreach \source/\dest in {
   b1/a2, a1/b1, a1/b2, b2/a2, b2/e1, e1/b1, e2/b1, b2/e2, b21/e2, e1/b11, e1/b12,
   b22/e2, c1/c11, c12/c1, c2/c22, c21/c2, d1/c2, c1/d1, c1/d2, d2/c2, f1/d1, d2/f2,
   d2/f1, f2/d1}
  \path[wedge] (\source) -- node[weight] {} (\dest);
 \end{scope}
\end{tikzpicture}
\end{center}
\caption{The digraph $\Delta^{(1)}_{2k}$, $k$ even}
\label{fig:Delta1e}
\end{figure}

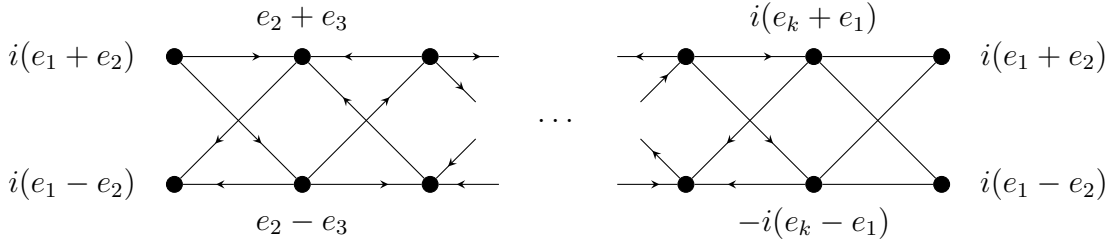
\begin{figure}
\begin{center}
\begin{tikzpicture}[auto, scale=1.7] % \Delta^{(1)}_{2k}$, $k$ odd
 \begin{scope}
  \foreach \type/\pos/\name in {
   {vertex/(0,0)/a2}, {vertex/(0,1)/a1}, {vertex/(1,1)/b1}, {vertex/(1,0)/b2},
   {vertex/(2,0)/e2}, {vertex/(2,1)/e1}, {empty/(2.6,1)/b11}, {empty/(2.6,0)/b21},
   {empty/(2.4,0.6)/b12}, {empty/(2.4,0.4)/b22}, {empty/(3.4,1)/c11}, {empty/(3.4,0)/c21},
   {empty/(3.6,0.6)/c12}, {empty/(3.6,0.4)/c22}, {vertex/(4,1)/c1}, {vertex/(4,0)/c2},
   {vertex/(5,1)/d1}, {vertex/(5,0)/d2}, {vertex/(6,1)/f1}, {vertex/(6,0)/f2}}
   \node[\type] (\name) at \pos {};
  \foreach \pos/\name in {
   {(3,0.5)/\dots}, {(-0.8,1)/i(e_1+e_2)}, {(-0.8,0)/i(e_1-e_2)}, {(1,1.3)/e_2+e_3},
   {(1,-0.3)/e_2-e_3}, {(5,1.3)/i(e_k+e_1)}, {(5,-0.3)/-i(e_k-e_1)}, {(6.8,1)/i(e_1+e_2)},
   {(6.8,0)/i(e_1-e_2)}}
   \node at \pos {$\name$};
  \foreach \edgetype/\source/ \dest in {
   wedge/b1/a2, wedge/a1/b1, wedge/a1/b2, wedge/b2/a2, wedge/b2/e1,
   wedge/e1/b1, wedge/e2/b1, wedge/b2/e2, wedge/b21/e2, wedge/e1/b11,
   wedge/e1/b12, wedge/b22/e2, wedge/c1/c11, wedge/c12/c1, wedge/c2/c22,
   wedge/c21/c2, wedge/d1/c2, wedge/c1/d1, wedge/c1/d2, wedge/d2/c2,
   pedge/f1/d1, pedge/d2/f2}
   \path[\edgetype] (\source) -- node[weight] {} (\dest);
 \end{scope}
 \begin{scope}
  \foreach \source/\dest in {d2/f1, f2/d1}
   \path[pedge] (\source) -- node[weight] {} (\dest);
 \end{scope}
\end{tikzpicture}
\end{center}
\caption{The digraph $\Delta^{(1)}_{2k}$, $k$ odd}
\label{fig:Delta1o}
\end{figure}

\begin{figure}
\begin{center}
\begin{tikzpicture}[auto, scale=1.7] % \Delta^{(i)}_{2k}$ with $k$ even
 \begin{scope}
  \foreach \type/\pos/\name in {
   {vertex/(0,0)/a2}, {vertex/(0,1)/a1}, {vertex/(1,1)/b1}, {vertex/(1,0)/b2},
   {vertex/(2,0)/e2}, {vertex/(2,1)/e1}, {empty/(2.6,1)/b11}, {empty/(2.6,0)/b21},
   {empty/(2.4,0.6)/b12}, {empty/(2.4,0.4)/b22}, {empty/(3.4,1)/c11}, {empty/(3.4,0)/c21},
   {empty/(3.6,0.6)/c12}, {empty/(3.6,0.4)/c22}, {vertex/(4,1)/c1}, {vertex/(4,0)/c2},
   {vertex/(5,1)/d1}, {vertex/(5,0)/d2}, {vertex/(6,1)/f1}, {vertex/(6,0)/f2}}
   \node[\type] (\name) at \pos {};
  \foreach \pos/\name in {
   {(3,0.5)/\dots}, {(-0.8,1)/{i}(e_1+e_2)}, {(-0.8,0)/{i}(e_1-e_2)}, {(1,1.3)/e_2+e_3},
   {(1,-0.3)/e_2-e_3}, {(3.9,1.3)/ie_{k-1}+ie_k}, {(3.9,-0.3)/-ie_{k-1}+ie_k},
   {(5.1,1.3)/ie_k+e_1}, {(5.1,-0.3)/ie_k-e_1},	{(6.8,1)/{i}(e_1+e_2)}, {(6.8,0)/{i}(e_1-e_2)}}
   \node at \pos {$\name$};
  \foreach \edgetype/\source/\dest in {
   wedge/b1/a2, wedge/a1/b1, wedge/a1/b2, wedge/b2/a2, wedge/b2/e1, wedge/e1/b1,
   wedge/e2/b1, wedge/b2/e2, wedge/b21/e2, wedge/e1/b11, wedge/e1/b12, wedge/b22/e2,
   wedge/c1/c11, wedge/c12/c1, wedge/c22/c2, wedge/c2/c21, pedge/c2/d1, pedge/d1/c1,
   pedge/d2/c1, pedge/c2/d2, wedge/f1/d1, wedge/d2/f2}
   \path[\edgetype] (\source) -- node[weight] {} (\dest);
 \end{scope}
 \begin{scope}
  \foreach \source/ \dest/\weight in {d2/f1, f2/d1}
   \path[wedge] (\source) -- node[weight] {} (\dest);
 \end{scope}
\end{tikzpicture}
\end{center}
\caption{The digraph $\Delta^{(i)}_{2k}$ with $k$ even}
\label{fig:Deltaie}
\end{figure}
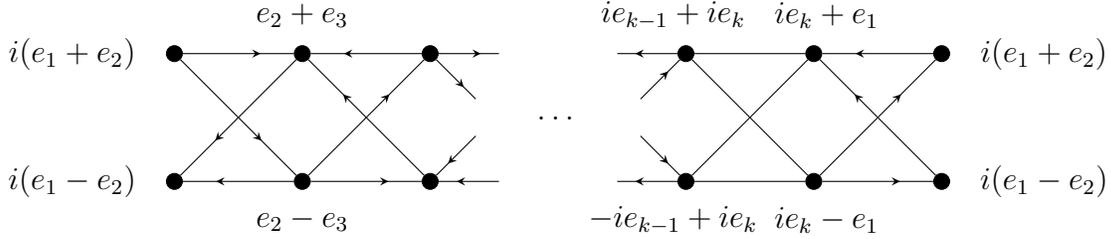

\begin{figure}
\begin{center}
\begin{tikzpicture}[auto, scale=1.7] % \Delta^{(i)}_{2k} k odd
 \begin{scope}
  \foreach \type/\pos/\name in {
   {vertex/(0,0)/a2}, {vertex/(0,1)/a1}, {vertex/(1,1)/b1}, {vertex/(1,0)/b2},
   {vertex/(2,0)/e2}, {vertex/(2,1)/e1}, {empty/(2.6,1)/b11}, {empty/(2.6,0)/b21},
   {empty/(2.4,0.6)/b12}, {empty/(2.4,0.4)/b22}, {empty/(3.4,1)/c11}, {empty/(3.4,0)/c21},
   {empty/(3.6,0.6)/c12}, {empty/(3.6,0.4)/c22}, {vertex/(4,1)/c1}, {vertex/(4,0)/c2},
   {vertex/(5,1)/d1}, {vertex/(5,0)/d2}, {vertex/(6,1)/f1}, {vertex/(6,0)/f2}}
   \node[\type] (\name) at \pos {};
  \foreach \pos/\name in {
   {(3,0.5)/\dots}, {(-0.8,1)/{i}(e_1+e_2)}, {(-0.8,0)/{i}(e_1-e_2)}, {(1,1.3)/e_2+e_3},
   {(1,-0.3)/e_2-e_3}, {(5,1.3)/ie_k+e_1}, {(5,-0.3)/ie_k-e_1}, {(6.8,1)/{i}(e_1+e_2)},
   {(6.8,0)/{i}(e_1-e_2)}}
   \node at \pos {$\name$};
  \foreach \source/\dest in {
   b1/a2, a1/b1, a1/b2, b2/a2, b2/e1, e1/b1, e2/b1, b2/e2, b21/e2, e1/b11, e1/b12,
   b22/e2, c11/c1, c1/c12, c22/c2, c2/c21, c2/d1, d1/c1, d2/c1, c2/d2, f1/d1, d2/f2}
   \path[wedge] (\source) -- node[weight] {} (\dest);
 \end{scope}
 \begin{scope}
  \foreach \source/\dest in {d2/f1, f2/d1}
   \path[wedge] (\source) -- node[weight] {} (\dest);
 \end{scope}
\end{tikzpicture}
\end{center}
\caption{The digraph $\Delta^{(i)}_{2k}$ with $k$ odd}
\label{fig:Deltaio}
\end{figure}
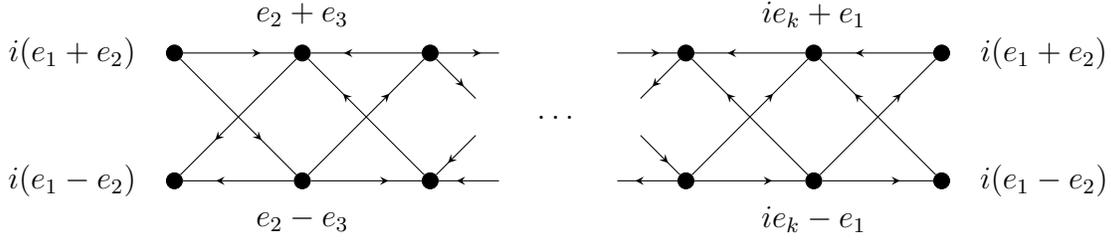

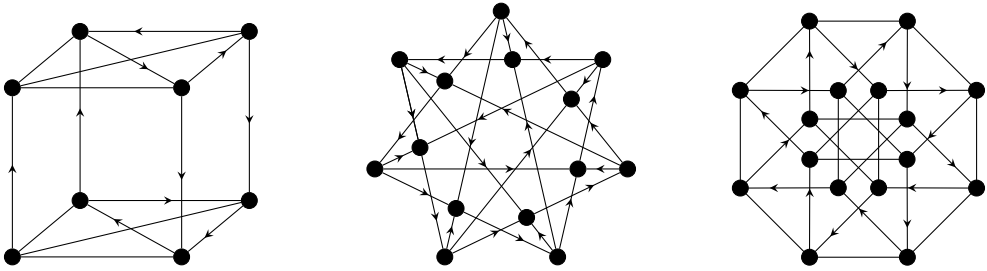
\begin{figure}
\begin{center}
\begin{tikzpicture}[auto, scale=1.5] % \Delta_8^\dagger
 \def\XthreeDadj{0.6}
 \def\YthreeDadj{0.5}
 \begin{scope}	
  \foreach \pos/\name in {
    {(0,0)/a}, {(0,1.5)/b}, {(1.5,0)/c}, {(1.5,1.5)/d},
    {(0 + \XthreeDadj,0 + \YthreeDadj)/e},
    {(0 + \XthreeDadj,1.5 + \YthreeDadj)/f},
    {(1.5 + \XthreeDadj,0 + \YthreeDadj)/g},
    {(1.5 + \XthreeDadj,1.5 + \YthreeDadj)/h}}
   \node[vertex] (\name) at \pos {};
  \foreach \edgetype/\source/\dest in {
   wmedge/a/b, pedge/a/c, pedge/a/g, pedge/b/h, pedge/e/a,
   wedge/g/c, wmedge/d/c, pedge/b/f, pedge/b/d, wmedge/e/f,
   wmedge/e/g, wmedge/h/g, wedge/h/f, wedge/d/h, wedge/f/d,
   wedge/c/e}
  \path[\edgetype] (\source) -- node[weight] {} (\dest);
 \end{scope}
\end{tikzpicture}
\hspace{1cm}
\begin{tikzpicture} % \Delta_{14}
 \begin{scope}[auto, scale=1.5]
  \foreach \pos/\name in {
   {(0,0)/a}, {(1,0)/b}, {(1.62,0.78)/c}, {(1.4,1.75)/d},
   {(0.5,2.18)/e}, {(-0.4,1.75)/f}, {(-0.62,0.78)/g},
   {(0.1,0.43)/t}, {(0.725,0.35)/u}, {(1.18,0.78)/v},
   {(1.122,1.4)/w}, {(0.6,1.75)/x}, {(0,1.56)/y}, {(-0.22,0.97)/z}}
   \node[vertex] (\name) at \pos {};
  \foreach \source/\dest in {
   a/t, e/t, z/a, f/z, a/w, d/w, a/u,
   t/b, b/u, f/u, b/x, e/x, b/v, u/c,
   c/v, g/v, c/y, f/y, c/w, v/d, d/z,
   g/z, d/x, w/e, e/y, x/f, f/z, y/g, g/t}
   \path[wedge] (\source) -- (\dest);
\end{scope}
\end{tikzpicture}
\hspace{1cm}
\begin{tikzpicture} % \Delta_{16}
 \newdimen\rad
 \rad=1.7cm
 \newdimen\radi
 \radi=0.7cm
 \foreach \x in {1,2,3,4,5,6,7,8}
  {
	\draw (22.5+\x*45:\radi) node[vertex] {};
	\draw[pedge] (22.5+\x*45:\radi) -- (157.5+\x*45:\radi);
  }
 \foreach \x in {1,2,3,4}
  {
   	\draw (-22.5+\x*90:\rad) node[vertex] {};
	\draw[pedge] (-22.5+\x*90:\rad)  -- (22.5+\x*90:\rad);
	\draw[wedge] ( 22.5+\x*90:\radi) -- (-22.5+\x*90:\rad);
	\draw[wedge] (-22.5+\x*90:\rad)  -- (-67.5+\x*90:\radi);
   }
 \foreach \x in {112.5,202.5,292.5,382.5}
  {
   	\draw (\x:\rad) node[vertex] {};
	\draw[pedge] (\x:\rad) -- (\x+45:\rad);
	\draw[wedge] (\x+45:\radi) -- (\x:\rad);
	\draw[wedge] (\x:\rad) -- (\x-45:\radi);
   }
\end{tikzpicture}
\end{center}
\caption{$\Delta_8^\dagger$, $\Delta_{14}$, $\Delta_{16}$}
\label{fig:Delta8dagger}
\end{figure}

\begin{figure}
\begin{center}
\begin{tikzpicture}
 \begin{scope}[auto, scale=0.7]
  \foreach \pos/\name in {
   (0:2)/a0, (45:2)/a1, (90:2)/a2, (135:2)/a3,
   (180:2)/a4, (225:2)/a5, (270:2)/a6, (315:2)/a7}
   \node[vertex] (\name) at \pos {};
  \foreach \source/\dest in {
   a0/a1, a1/a2, a2/a3, a3/a4, a4/a5, a5/a6, a6/a7, a7/a0}
   \path[wedge] (\source) -- (\dest);
 \end{scope}
\end{tikzpicture}
\quad
\begin{tikzpicture}
 \begin{scope}[auto, scale=0.7]
  \foreach \pos/\name in {
   (0:2)/a0, (45:2)/a1, (90:2)/a2, (135:2)/a3,
   (180:2)/a4, (225:2)/a5, (270:2)/a6, (315:2)/a7}
   \node[vertex] (\name) at \pos {};
  \foreach \edgetype\source/\dest in {
   wedge/a0/a1, wedge/a1/a2, wedge/a2/a3, wedge/a3/a4,
   wedge/a4/a5, wedge/a5/a6, wedge/a6/a7, dwedge/a0/a7}
   \path[\edgetype] (\source) -- (\dest);
 \end{scope}
\end{tikzpicture}
\quad
\begin{tikzpicture}
 \begin{scope}[auto, scale=0.7]
  \foreach \pos/\name in {
   (0:2)/a0, (45:2)/a1, (90:2)/a2, (135:2)/a3,
   (180:2)/a4, (225:2)/a5, (270:2)/a6, (315:2)/a7}
   \node[vertex] (\name) at \pos {};
  \foreach \edgetype/\source/\dest in {
   dpedge/a0/a1, wedge/a1/a2, wedge/a2/a3, wedge/a3/a4,
   wedge/a4/a5, wedge/a5/a6, wedge/a6/a7, wedge/a7/a0}
   \path[\edgetype] (\source) -- (\dest);
 \end{scope}
\end{tikzpicture}
\quad
\begin{tikzpicture}
 \begin{scope}[auto, scale=0.7]
  \foreach \pos/\name in {
   (0:2)/a0, (45:2)/a1, (90:2)/a2, (135:2)/a3,
   (180:2)/a4, (225:2)/a5, (270:2)/a6, (315:2)/a7}
   \node[vertex] (\name) at \pos {};
  \foreach \edgetype/\source/\dest in {
   dpedge/a0/a1, wedge/a1/a2, wedge/a2/a3, wedge/a3/a4,
   wedge/a4/a5, wedge/a5/a6, wedge/a6/a7, dwedge/a0/a7}
   \path[\edgetype] (\source) -- (\dest);
 \end{scope}
\end{tikzpicture}
\end{center}
\caption{$D_{n},\tilde{C}_n,\tilde{C}'_n,\tilde{C}''_n$
for $n=8$}
\label{fig:DCCC}
\end{figure}
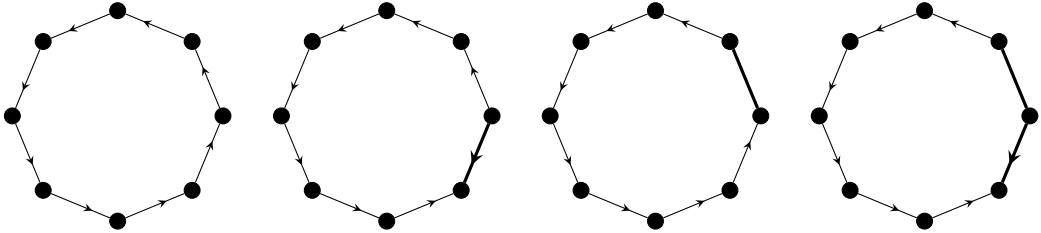

In order to state our theorem, we need to define a number
of graphs.
The graphs $\Delta_{2k}^{(1)}$,
$\Delta_{2k}^{(i)}$,
$\Delta_{8}^\dagger,\Delta_{14}$ or $\Delta_{16}$
are depicted in
Fig.~\ref{fig:Delta1e}--\ref{fig:Delta8dagger}.
The labels attached to the vertices in
Fig.~\ref{fig:Delta1e}--\ref{fig:Deltaio} will be explained in
Section~\ref{sec:3}.
We denote by $P_n$ the path graph on $n$ vertices, and by $C_n$
the (undirected) cycle graph on $n$ vertices.
We denote by $D_n$ the directed cycle on $n$ vertices.
Let  $\tilde{C}_n$ be the digraph obtained from $D_n$ by reversing the direction
of one of the arcs.
Let  $\tilde{C}'_n$ be the digraph obtained from $D_n$ by replacing
one of the arcs by a digon.
Let  $\tilde{C}''_n$ be the digraph obtained from $D_n$ by
taking two consecutive arcs and then replacing the first one by a digon
and reversing the direction of the second.
See Figure~\ref{fig:DCCC} for illustrations of these four graphs.
For positive integers $a,b,c$, let $Y_{a,b,c}$ be the tree obtained by
taking paths
$P_{a+1},P_{b+1},P_{c+1}$
and identifying an end vertex of each into a single vertex.
Note that $Y_{a,b,c}$ has $a + b + c + 1$ vertices.
For nonnegative integers $a_1,a_2,a_3,a_4$, let
$\square_{a_1,a_2,a_3,a_4}$
be a digraph obtained from $\tilde{C}_4$
with (consecutive) vertices $v_1,v_2,v_3,v_4$ by
adding directed paths on $a_i+1$ vertices
that are attached to $v_i$ for each $i=1,2,3,4$.
The digraphs
$\tilde{U}_1$ and $\tilde{U}_6$ can be found in
Fig.~\ref{fig:U1U6}.
The bipartite signed graphs
$U_1,\dots,U_{11}$ are depicted in
Fig.~\ref{fig:F-U}.

By a subdigraph of a digraph $\Delta=(V,E)$, we mean
a digraph of the form $(W,E\cap(W\times W))$, where $W$ is
a subset of $V$. It is clear that the Hermitian adjacency
matrix of a subdigraph of $\Delta$ is a principal submatrix
of $H(\Delta)$.
See Definition~\ref{dfn:canonical} for a definition
of canonical digraphs of a bipartite signed graph.
Our main result is the following theorem.
Since switching equivalence (see Definition~\ref{dfn:sw})
preserves Hermitian spectra of digraphs,
we give our classification
of connected digraphs whose Hermitian
spectral radius is at most $2$ up to switching equivalence.

\begin{theorem}\label{thm:main}
Let $\Delta$ be a  connected digraph. If $\rho(\Delta)\leq2$,
then $\Delta$ is switching equivalent to a subdigraph of one of the following:
\begin{enumerate}
\item\label{i:t1}
$\Delta_{2k}^{(1)}$,
%a toral tesselation
%(see Fig.~\ref{fig:Delta1e},\ref{fig:Delta1o}),
\item\label{i:t2}
$\Delta_{2k}^{(i)}$,
%a toral tesselation
%(see Fig.~\ref{fig:Deltaie},\ref{fig:Deltaio}),
\item\label{i:sp1}
$\Delta_{8}^\dagger,\Delta_{14}$ or $\Delta_{16}$.
%sporadic
%(see Fig.~\ref{fig:Delta8dagger}).
%\item\label{i:s1}
%subdigraphs of {\rm(\ref{i:t1})}, {\rm(\ref{i:t2})} and {\rm(\ref{i:sp1})}
\end{enumerate}
Moreover, if $\rho(\Delta)<2$,
then $\Delta$ is switching equivalent to a subdigraph of one of the following:
\begin{enumerate}
\setcounter{enumi}{3}
\item\label{i:Dn} $D_n$, where $n\not\equiv0\pmod{4}$,
\item\label{i:Cn} $\tilde{C}_n$, where $n\not\equiv2\pmod{4}$,
\item\label{i:Cn1} $\tilde{C}'_n$, where $n\not\equiv1\pmod{4}$,
\item\label{i:Cn2} $\tilde{C}''_n$, where $n\not\equiv3\pmod{4}$,
\item\label{i:path} a path $P_n$,
\item\label{i:sq} $\square_{a,0,c,0}$ for some $a,c\geq0$,
\item\label{i:Y} $Y_{a,1,1}$ for some $a\geq1$,
\item\label{i:U1} $\tilde{U}_1$,
% (see Fig.~\ref{fig:U1U6}),
\item\label{i:U6} $\tilde{U}_6$,
% (see Fig.~\ref{fig:U1U6}),
%\item the digraph obtained from the directed triangle $D_3$
%by adding a vertex and an arc from this vertex to one of the vertices of $D_3$,
\item\label{i:canonical}
canonical digraphs of the bipartite signed graphs
$U_1,\dots,U_{11}$.
% (see Figure~\ref{fig:F-U}).
%(maximal subject to $\rho<2$ only)
%\item
%subdigraph of (iv)
\end{enumerate}
Conversely, every digraph in \ref{i:t1}--\ref{i:sp1}
%\textup{(i)}--\textup{(iii)}
has spectral radius $2$, and every digraph in
\ref{i:Dn}--\ref{i:canonical}
%\textup{(iv)}--\textup{(xiii)}
has  spectral radius less than $2$.
\end{theorem}

There are a number of proper subdigraphs of $\Delta_{2k}^{(1)}$ and
$\Delta_{2k}^{(i)}$ having  spectral radius $2$.
A partial classification of maximal digraphs with spectral radius
at most $2$ has been obtained in \cite{Y} under the assumption
that the underlying graph is $C_4$-free.

Note that the  spectral radius is
monotone with respect to taking subdigraphs: if $\Delta'$
is a subdigraph of $\Delta$,
then $\rho(\Delta')\leq \rho(\Delta)$.
This reduces the problem to consideration of a number of forbidden subdigraphs,
and this led to the classification in \cite{GM2}.

We shall exploit another approach in this paper.
Note that symmetric (Hermitian) matrices over a ring of algebraic integers
having all their eigenvalues in the interval $[-2, 2]$ are called {\bf cyclotomic},
and are of independent interest in algebraic number theory.
Integer cyclotomic matrices were described by McKee and Smyth \cite{MS},
whose proof also involved the classical root systems.
Their result was further extended by Greaves \cite{G} to Hermitian cyclotomic matrices
over the Eisenstein and Gaussian integers.

Since a Hermitian adjacency matrix is a Hermitian matrix over the Gaussian integers,
the results of Greaves \cite{G} contain, however, do not immediately imply those of Guo and Mohar \cite{GM2}
due to the following obstacles.
First of all, the classification in \cite{G} is given
up to equivalence,
which is weaker than switching equivalence
for digraphs (see Section~\ref{sec:pre}).
Secondly, the classification in \cite{G} does not explicitly list matrices
with all their eigenvalues in the {\it open} interval $(-2, 2)$.

We proceed as follows. We first show that each of $\mathbb{Z}[i]$-graphs, corresponding
to maximal indecomposable Hermitian cyclotomic matrices
determined in \cite{G},
gives rise to a unique switching equivalence class of digraphs.
This implies that
every digraph with  spectral radius at most $2$
is switching equivalent to a subdigraph of one of the digraphs in
Theorem~\ref{thm:main} \ref{i:t1}--\ref{i:sp1}
%(i)--(iii).
This classification contains digraphs with
 spectral radius less than $2$
classified by Guo and Mohar \cite{GM2}, so in principle,
one can try to derive their result by looking at
subdigraphs of digraphs with  spectral radius $2$.
Instead, we show that such a classification follows from the results of McKee and Smyth \cite{MS}
via the notion of the associated signed graph
of a digraph, which we introduce in Section~\ref{sec:pre}.
In doing so, we found a counterexample to the statement of \cite[Lemma~4.8(b)]{GM2},
leading to an omission in \cite[Theorem 4.15]{GM2}.
We thus complete the statement of \cite[Theorem 4.15]{GM2}
by supplying the missing digraph which is the canonical
digraph of $U_7$ (see Fig.~\ref{fig:CU7})
in our Theorem~\ref{thm:main} \ref{i:canonical}.

After giving preliminaries in Section~\ref{sec:pre},
we give a proof of the first part of Theorem~\ref{thm:main}
in Section~\ref{sec:3}. The second part of Theorem~\ref{thm:main}
is proved in Section~\ref{sec:4}.
Finally, in Section~\ref{sec:cr}, we
establish a correspondence
between the digraphs with  spectral radius at most $2$
and the Gaussian root lattices. We also
characterize digraphs with smallest eigenvalue greater than $-\sqrt{2}$,
which strengthens \cite[Proposition~8.6]{GM}.

\section{Preliminaries}
\label{sec:pre}

\subsection{Equivalence relations on matrices}
\label{subsec:2.1}

In this subsection, we discuss equivalence relations defined
in \cite{G,M}. Let $n$ be a positive integer, and let
$\cH_n$ denote the set of all Hermitian matrices of order $n$
with entries in $\{0,\pm1,\pm i\}$, where $i=\sqrt{-1}$.
Let $U_n(\Z[i])$ be the
subgroup of the complex unitary group consisting of
those matrices which have entries in $\Z[i]$.
Then
$U_n(\Z[i])$ is generated by the permutation matrices
together with the
diagonal matrices with diagonal entries in $\{\pm1,\pm i\}$.
For two matrices $A,B\in\cH_n$, we say that $A$ is \textbf{strongly equivalent}
to $B$, written $A\approx B$,
 if $A=QBQ^*$ or $A=Q\overline{B}Q^*$ for some $Q\in U_n(\Z[i])$.
The matrices $A$ and $B$ are called \textbf{equivalent},
 written $A\sim B$, if $A$ is
strongly equivalent to
$B$ or $-B$.

\begin{dfn}\label{dfn:sw}
We say that a digraph $\Delta'$ is obtained from $\Delta$ by
\textbf{four-way switching} if
$H(\Delta')=QH(\Delta)Q^*$ for some $Q\in U_n(\Z[i])$.
The digraph whose Hermitian adjacency matrix is $\overline{H(\Delta)}$
is called the \textbf{converse} of a digraph $\Delta$.
We say that two digraphs
are \textbf{switching equivalent} (see \cite{M})
if one can be obtained from the other by a sequence of four-way switchings and operations of taking the converse.
\end{dfn}

We remark that the four-way switching
was defined in \cite{GM} by modifying the set of arcs
with respect to a certain partition of the vertex set.
Our definition is equivalent to the one in \cite{GM}.

\begin{lemma}\label{lem:4way}
Let $\Delta_1$ and $\Delta_2$ be digraphs with
respective Hermitian adjacency
matrices $H_1$ and $H_2$. Then the following statements are equivalent:
\begin{enumerate}
\item $\Delta_1$ and $\Delta_2$ are switching equivalent,
\item $H_1$ and $H_2$ are strongly equivalent.
\end{enumerate}
\end{lemma}
\begin{proof}
Immediate from the definitions.
\end{proof}

\subsection{Associated signed graphs}
\label{subsec:2.2}

A \textbf{signed graph}
$S$ is a triple $(V, E^{+}, E^{-})$
of a set $V$ of vertices,
a set $E^{+}$ of $2$-subsets
of $V$ (called \textbf{positive} edges), and
a set $E^{-}$ of $2$-subsets
of $V$ (called  \textbf{negative} edges)
such that
$E^{+} \cap E^{-} = \emptyset$.
We depict positive (resp.\ negative) edges by solid (resp.\ dashed) lines
(see Fig.~\ref{fig:QO}, \ref{fig:F-U}).
The adjacency matrix of a signed graph $S=(V, E^{+}, E^{-})$
is the matrix $A(S)$ whose rows and columns are indexed by $V$
such that its $(x,y)$-entry is $1,-1,0$ according to
$\{x,y\}\in E^+,E^-$, otherwise, respectively.
We say two signed graphs $S,S'$ are \textbf{strongly equivalent}
(resp.\ \textbf{equivalent}) if $A(S)\approx A(S')$
(resp.\ $A(S)\sim A(S')$).
By a subgraph of a signed graph $S$, we mean
a signed graph of the form $(W,E^+_W,E^-_W)$, where $W$ is
a subset of $V$, and $E^\pm_W$ is the subset of $E^\pm$
consisting of those $2$-subsets that are contained in $W$.
It is clear that the adjacency
matrix of a subgraph of $S$ is a principal submatrix
of $A(S)$.

Given a connected digraph $\Delta$ with Hermitian adjacency matrix
$H=A+iB$, where $A$ is a symmetric $(0,1)$-matrix and $B$ is a
skew symmetric $(0,\pm1)$-matrix,
the \textbf{associated signed graph} $\cS(\Delta)$ of $\Delta$ is
the signed graph with adjacency matrix
\begin{equation}\label{1a}
C=\begin{bmatrix}
A&B\\ B^\top&A
\end{bmatrix}.
\end{equation}

\begin{lemma}\label{lem:SDspec}
For a digraph $\Delta$,
the Hermitian spectrum of $\Delta$ is the same as that of the $\cS(\Delta)$
in which multiplicities are doubled.
\end{lemma}
\begin{proof}
Since $H$ is Hermitian, we have
\[\begin{bmatrix} I&0\\-iI&I\end{bmatrix}
\begin{bmatrix} A&B\\B^\top&A\end{bmatrix}
\begin{bmatrix} I&0\\iI&I\end{bmatrix}
=\begin{bmatrix} A+iB&B\\0&A-iB\end{bmatrix}.\]
This implies
\[\det(xI-C)=(\det(xI-H))^2,\]
and hence the eigenvalues of $C$ are the same as those of $H$, with
multiplicities doubled.
\end{proof}

We can characterize $\cS(\Delta)=(V_1\cup V_2,E^+\cup E^-)$ as follows.
The vertex set $V_1\cup V_2$
is the disjoint union of two copies $V_i=\{x_i\mid x\in V(\Delta)\}$ of
the vertex set $V(\Delta)$ of the digraph $\Delta$.
The edges of $\cS(\Delta)$ are the following:
\begin{enumerate}
\item If $\{x,y\}$ is a digon in $\Delta$, then
$\{x_1,y_1\},\{x_2,y_2\}\in E^+$;
\item If $(x,y)$ is an arc in $\Delta$, then
$\{x_1,y_2\}\in E^+$ and $\{x_2,y_1\}\in E^-$.
\end{enumerate}

For a digraph $\Delta$ (or a signed graph $S$),
let $G(\Delta)$ ($G(S)$, respectively) denote
the \textbf{underlying graph} of $\Delta$ (or $S$, respectively), i.e.,
a graph obtained from $\Delta$ (from $S$)
by replacing all of its arcs (signed edges, respectively) with undirected edges.
We say a digraph (or a signed graph) is connected if
its underlying graph is connected.
For a vertex $x$ of a graph $G$, we denote by
$\deg_G(x)$ the degree of $x$.

\begin{lemma}\label{lem:asg2}
With reference to the above description of $\cS(\Delta)$,
%If a connected graph $G$ is isomorphic to the underlying graph
%of the associated signed graph $\cS(\Delta)$ for some digraph $\Delta$,
%then the vertex set $V(G)$ is the disjoint union of $2$-cocliques
%$\bigcup_{x\in X}\{x_+,x_-\}$ such that the following conditions hold:
the following statements hold for all $x\in V(\Delta)$:
\begin{enumerate}
\item\label{sg2_1}
the vertices $x_1$ and $x_2$ are not adjacent in $G(\cS(\Delta))$,
\item\label{sg2_2}
the vertices $x_1$ and $x_2$ have no common neighbors in
$G(\cS(\Delta))$,
\item\label{sg2_3}
$\deg_{G(\cS(\Delta))}(x_1)=\deg_{G(\cS(\Delta))}(x_2)=\deg_{G(\Delta)}(x)$.
\end{enumerate}
\end{lemma}
\begin{proof}
Immediate.
\end{proof}

\begin{lemma}\label{lem:1}
The associated signed graph of a connected digraph $\Delta$
is either connected, or has two connected components, say $S_1$ and $S_2$.
The former case occurs precisely when there is a cycle in $\Delta$
containing an odd number of arcs.
In the latter case, $H(\Delta)$, $A(S_1)$ and $A(S_2)$ are strongly equivalent.
\end{lemma}
\begin{proof}
Note that $\cS(\Delta)$
is connected if and only if there exists a vertex $x$ of $\Delta$ such that
$x_1$ and $x_2$ are connected by a path in $\cS(\Delta)$.
This condition
is equivalent to the existence of a cycle containing an odd number of arcs.

Suppose that $\cS(\Delta)$ is disconnected, and let
$S_1$ be a connected component of $\cS(\Delta)$.
Let $A_1,A_2$ and $B_1$ denote the submatrix of \eqref{1a} corresponding to
$(V_1\cap S_1)^2$, $(V_1\setminus S_1)^2$ and
$(V_1\cap S_1)\times (V_2\cap S_1)$, respectively.
Since $B=-B^\top$ in \eqref{1a},
the matrix \eqref{1a} has the form
\[\begin{bmatrix}
A_1&0&0&B_1\\
0&A_2&-B_1^\top&0\\
0&-B_1&A_1&0\\
B_1^\top&0&0&A_2\end{bmatrix}.\]
This implies that $S_1$
and the other connected component have adjacency matrix
\[\begin{bmatrix}A_1&\pm B_1\\ \pm B_1^\top&A_2\end{bmatrix}\]
which are strongly equivalent to the Hermitian adjacency matrix
\[\begin{bmatrix}A_1&iB_1\\ -iB_1^\top&A_2\end{bmatrix}\]
of $\Delta$.
\end{proof}

For example, the associated signed graph of $\Delta_8^\dagger$
(see Figure~\ref{fig:Delta8dagger}) is connected, since
$\Delta_8^\dagger$ contains a directed triangle.
The associated signed graphs of $\Delta_{14}$ and $\Delta_{16}$
(see Figure~\ref{fig:Delta8dagger}) have two connected components
since, after removing digons, the underlying graphs are
bipartite.

\begin{dfn}\label{dfn:canonical}
Let $S$ be a connected bipartite signed graph with
bipartition $V(S)=V_1\cup V_2$. Let $D$ be the diagonal
matrix whose rows and columns are indexed by $V(S)$, and
whose $(x,x)$-entry is $1$ or $i$ according to $x\in V_1$ or $V_2$.
Then $D^*A(S)D$ is the Hermitian adjacency matrix of a digraph
$\Delta$. We call $\Delta$ a \textbf{canonical digraph}
of the bipartite signed graph $S$.
\end{dfn}

As an example, a canonical digraph of the bipartite signed graph
$U_7$ in Fig.~\ref{fig:F-U} is shown in Fig.~\ref{fig:CU7}.
A canonical digraph is not uniquely determined, but it is unique
up to switching equivalence.

Since the underlying graph of a canonical digraph
$\Delta$ of a bipartite signed graph is bipartite,
Lemma~\ref{lem:1} shows that the associated signed graph
of $\Delta$ is disconnected.

\begin{figure}[h]
\begin{center}
\begin{tikzpicture}[auto, scale=1]
 \foreach \type/\pos/\name in {
  {vertex/(0,1)/90},{vertex/(0,-1)/270},
  {vertex/(0.86,0.5)/30}, {vertex/(0.86,-0.5)/-30},
  {vertex/(-0.86,0.5)/150}, {vertex/(-0.86,-0.5)/210},
  {vertex/(-1.86,0.5)/y+}, {vertex/(-1.86,-0.5)/y-}}
  \node[\type] (\name) at \pos {};
 \foreach \edgetype/\source/ \dest in {
  wedge/-30/30, wedge/90/30, wedge/90/150, wedge/150/210,
  wedge/210/270, wedge/-30/270, wedge/y+/150, wedge/y+/y-, wedge/210/y-}
  \path[\edgetype] (\source) -- (\dest);
\end{tikzpicture}
\end{center}
\caption{A canonical digraph of $U_7$}
\label{fig:CU7}
\end{figure}
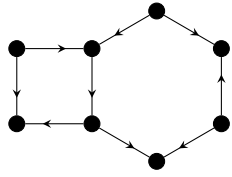

\begin{lemma}\label{lem:bipartite0}
Let $S$ be a connected bipartite signed graph.
For a digraph $\Delta$, $H(\Delta)$ is strongly equivalent to
$A(S)$ if and only if $\Delta$ is switching equivalent to a canonical digraph of $S$.
\end{lemma}
\begin{proof}
Let $\Delta_0$ be a canonical digraph of $S$.
If $\Delta$ is switching equivalent to $\Delta_0$,
then $H(\Delta)\approx H(\Delta_0)\approx A(S)$ by
Lemma~\ref{lem:4way} and construction.

Conversely, suppose $H(\Delta)\approx A(S)$.
Since $S$ is bipartite, every cycle in
a canonical digraph $\Delta_0$ of $S$ has an even number of arcs.
By Lemma~\ref{lem:1}, $\cS(\Delta_0)$
has two connected components, say $S_1$ and $S_2$. It is easy to see
that $A(S)\approx A(S_i)$ for $i=1,2$, and hence by Lemma~\ref{lem:1},
we have
$H(\Delta_0)\approx A(S_1)\approx A(S)\approx H(\Delta)$.
Thus, $\Delta$ is switching equivalent to $\Delta_0$ by
Lemma~\ref{lem:4way}.
\end{proof}

\begin{lemma}\label{lem:bipartite}
Let $U$  be a bipartite signed graph, and let $\Delta_0$ be a
canonical digraph of $U$.
Let $\Delta$ be a digraph such that $\cS(\Delta)$ is disconnected,
and $\cS(\Delta)$ has a connected component which is equivalent to
some subgraph of $U$.
Then $\Delta$ is switching equivalent to some subdigraph
of $\Delta_0$.
\end{lemma}
\begin{proof}
By the assumption, $\cS(\Delta)$ has a connected component $S_1$
which is equivalent to some subgraph $U'$ of $U$.
Since $U$ is bipartite, so is $U'$. Thus $A(U')$ is strongly equivalent to
$-A(U')$. This implies that $A(S_1)\approx A(U')$.
Since $A(U)\approx H(\Delta_0)$ by Lemma~\ref{lem:bipartite0},
we have $A(U')\approx H(\Delta'_0)$ for some subdigraph $\Delta'_0$ of $\Delta_0$.
Since $H(\Delta)\approx A(S_1)$ by Lemma~\ref{lem:1},
we conclude  $H(\Delta)\approx H(\Delta'_0)$. The result follows from
Lemma~\ref{lem:4way}.
\end{proof}

\section{Spectral radius at most $2$}\label{sec:3}

Hermitian adjacency matrices with spectral radius at most $2$ are
considered in a broader context in \cite{G}.
Greaves \cite[Theorem~3.1]{G} classified,
among others,
cyclotomic matrices
over Gaussian integers with unit entries and zero diagonals,
under the assumption of maximality and
indecomposability.
It is also shown in \cite[Theorem~3.4]{G} that
every indecomposable cyclotomic matrix over Gaussian integers is
a principal submatrix of a maximal one.
These matrices are described in terms of weighted digraphs in \cite{G},
which are not digraphs in our sense.
Alternatively, we can describe each of these matrices $H$ by giving a set of
vectors whose Gram matrix is $H+2I$.
For convenience, if $H+2I$ is the Gram matrix of a set of vectors
each of which has squared norm $2$,
then we call $H$ the \textbf{displaced Gram matrix} of this set.
We denote by $e_1,e_2,\dots,e_n$, the standard orthonormal
basis of the vector space $\C^n$.

For a positive integer $k$, we define a matrix
$T^{(1)}_{2k}$ as
%in such a way that $T_{2k}^{(1)}+2I$ is the Gram
the displaced Gram matrix of the
set of vectors
\[\{e_p\pm e_{p+1}\mid 1\leq p\leq k\},\]
where indices are read modulo $k$.
We also define a matrix
$T^{(i)}_{2k}$ as the
%in such a way that $T_{2k}^{(i)}+2I$ is the
displaced Gram matrix of the
set of vectors
\[\{e_p\pm e_{p+1}\mid 1\leq p< k\}\cup\{ie_k\pm e_1\}.\]
We define three more matrices
$S_8^\dag$, $S_{14}$, and $S_{16}$, as follows:
%The $\Z[i]$-graph $S_8^\dagger$ has adjacency matrix
\[S_8^\dag=\begin{bmatrix}
0&-1&-1&i&1&0&0&0\\
-1&0&i&-1&0&1&0&0\\
-1&-i&0&1&0&0&1&0\\
-i&-1&1&0&0&0&0&1\\
1&0&0&0&0&1&1&-i\\
0&1&0&0&1&0&-i&1\\
0&0&1&0&1&i&0&-1\\
0&0&0&1&i&1&-1&0
\end{bmatrix},\]
\[S_{14}=\begin{bmatrix}
0&M\\ M^\top&0\end{bmatrix},\]
\[S_{16}=\begin{bmatrix}
C+C^\top&-C+C^\top\\ C-C^\top&C^3+C^5\end{bmatrix},\]
where $M$ and $C$ are the circulant matrices of order $7$ and $8$,
with first row
$[1,1,0,1,0,0,-1]$ and
$[0,1,0,\dots,0]$, respectively.

\begin{theorem}[{\cite[Theorems~3.1 and 3.4]{G}}]\label{thm:Gary}
Let $H$ be an indecomposable Hermitian matrix
with spectral radius at most $2$
that has only zeros on the diagonal and whose nonzero entries
are from $\{\pm1,\pm i\}$.
Then $H$ is equivalent to a principal submatrix of one of the matrices
$T_{2k}^{(1)}$, $T_{2k}^{(i)}$, $S_8^\dag$, $S_{14}$, and $S_{16}$,
which are maximal subject to these conditions.
\end{theorem}

The reader might think that Theorem~\ref{thm:Gary} immediately implies
the classification of maximal digraphs with
spectral radius at most $2$
up to switching equivalence. In view of Lemma~\ref{lem:4way}, however,
switching equivalence of digraphs amounts to strong equivalence of
their Hermitian adjacency matrices. Since Theorem~\ref{thm:Gary}
classifies possible matrices up to equivalence, but not strong equivalence,
$H$ may not be strongly
equivalent to $-H$ for those $H$ appearing in Theorem~\ref{thm:Gary}.
We shall show however, that for
each of the matrices
$H$ appearing in Theorem~\ref{thm:Gary}, $H$ turns out to be strongly
equivalent to $-H$. Moreover, we shall show that
$H$ is strongly equivalent to
the Hermitian adjacency matrix of a digraph.

Recall that the digraphs
$\Delta_{2k}^{(1)}$ and $\Delta_{2k}^{(i)}$
are defined in Fig.~\ref{fig:Delta1e}--\ref{fig:Deltaio}.

\begin{proposition}\label{prop:T}
Let $x=1$ or $i$.
Every Hermitian matrix
which is equivalent to $T_{2k}^{(x)}$ is strongly equivalent
to $H(\Delta_{2k}^{(x)})$.
\end{proposition}
\begin{proof}
Observe that the matrix $T_{2k}^{(x)}$ is of the form
\[\begin{bmatrix}
A&B\\B&-A
\end{bmatrix}.\]
Then $QT_{2k}^{(x)}Q^*=-T_{2k}^{(x)}$ for
\[Q=\begin{bmatrix}
0&I\\-I&0
\end{bmatrix}.
\]
Thus, it suffices to show that $T_{2k}^{(x)}$ is strongly equivalent to
$H(\Delta_{2k}^{(x)})$.

Suppose first $x=1$.
If $k$ is even, then $T_{2k}^{(1)}$ is strongly equivalent to
the displaced Gram matrix of the set of vectors
\begin{align*}
&\{e_p\pm e_{p+1}\mid 1\leq p\leq k,\; p\text{ even}\}\cup
\{i(e_p\pm e_{p+1})\mid 1\leq p\leq k,\; p\text{ odd}\}
\end{align*}
and this is the  matrix $H(\Delta_{2k}^{(1)})$
(see Fig.~\ref{fig:Delta1e}).
Next suppose $k$ is odd. Then $T_{2k}^{(1)}$ is strongly equivalent to
the displaced Gram matrix of the set of vectors
\begin{align*}
&\{e_p\pm e_{p+1}\mid 1\leq p\leq k,\; p\text{ even}\}\cup
\{i(e_p\pm e_{p+1})\mid 1\leq p< k,\; p\text{ odd}\}
\\&\cup\{i(e_k+e_1)\}\cup\{-i(e_k-e_1)\},
\end{align*}
and this is the  matrix $H(\Delta_{2k}^{(1)})$
(see Fig.~\ref{fig:Delta1o}).

Next suppose $x=i$.
If $k$ is odd, then
Then $T_{2k}^{(i)}$ is strongly equivalent to
the displaced Gram matrix of the set of vectors
\[\{e_p\pm e_{p+1}\mid 1\leq p\leq k,\; p\text{ even}\}\cup
\{i(e_p\pm e_{p+1})\mid 1\leq p< k,\; p\text{ odd}\}
\cup\{ie_k\pm e_1\},\]
and this is the  matrix $H(\Delta_{2k}^{(i)})$
(see Fig.~\ref{fig:Deltaio}).
If $k$ is even, then $T_{2k}^{(i)}$ is strongly equivalent to
the displaced Gram matrix of the set of vectors
\begin{align*}
&\{e_p\pm e_{p+1}\mid 1\leq p\leq k,\; p\text{ even}\}\cup
\{i(e_p\pm e_{p+1})\mid 1\leq p< k,\; p\text{ odd}\}
\\&\cup\{\pm ie_{k-1}+ie_k\}\cup\{ie_k\pm e_1\},
\end{align*}
and this is the  matrix $H(\Delta_{2k}^{(i)})$
(see Fig.~\ref{fig:Deltaie}).
\end{proof}

\begin{proposition}\label{prop:S}
For $S=S_8^\dagger,S_{14}$ or $S_{16}$,
every Hermitian matrix
which is equivalent to $S$ is strongly equivalent
to $H(\Delta_{8}^\dagger),H(\Delta_{14})$ or $H(\Delta_{16})$,
respectively.
\end{proposition}
\begin{proof}
For each $S=S_8^\dagger,S_{14}$ or $S_{16}$,
it suffices to show the Hermitian adjacency matrix
of the corresponding digraph is strongly equivalent to $S$ and $-S$.

Let $D_1$ and $D_2$ be the diagonal matrices with diagonal entries
\[[-i,i,i,1,1,1,1,-i]\text{ and }
[ 1, 1, 1, -i, i, -i, -i, -1 ],\]
respectively. Then $D_1S_8^\dagger D_1^*=-D_2^*\overline{S_8^\dagger}D_2$
is the Hermitian adjacency matrix of $\Delta_8^\dagger$.

Let $D_1$ be the diagonal matrix with diagonal entries
$[1,\dots,1,i,\dots,i]$, where $1$ and $i$ are repeated $7$ times each.
Then $D_1S_{14}D_1^*=-D_1^*S_{14}D_1$ is the Hermitian adjacency
matrix of $\Delta_{14}$.

Let $D_1$ and $D_2$ be the diagonal matrices
with diagonal entries
\[[i,\dots,i,1,\dots,1],\text{ and }
[i,-i,i,-i,i,-i,i,-i,1,-1,1,-1,1,-1,1,-1],\]
respectively. Then $D_1S_{16}D_1^*=-D_2S_{16}D_2^*$ is the Hermitian adjacency
matrix of $\Delta_{16}$.
\end{proof}

Let $\Delta$ be a connected digraph with $\rho(\Delta)\leq2$. Let
$H$ be the Hermitian adjacency matrix of $\Delta$. Then by
Theorem~\ref{thm:Gary}, $H$ is equivalent to a principal submatrix of
one of the matrices
\begin{equation}\label{eq:T}
T_{2k}^{(1)},\quad T_{2k}^{(i)},\quad S_8^\dag,\quad S_{14},
\text{ or }S_{16}.
\end{equation}
This implies that $H$ is a principal submatrix of a
Hermitian matrix $T$ which is equivalent to one of the matrices in \eqref{eq:T}.
By Propositions~\ref{prop:T} and \ref{prop:S},
$T$ is strongly equivalent to one of the matrices
\begin{equation}\label{eq:H}
H(\Delta_{2k}^{(1)}),\quad
H(\Delta_{2k}^{(i)}),\quad H(\Delta_8^\dag),\quad H(\Delta_{14})
\text{ or }H(\Delta_{16}).
\end{equation}
This implies that $H$ is strongly equivalent to a principal submatrix
of one of the matrices in \eqref{eq:H}.
By Lemma~\ref{lem:4way},
$\Delta$ is switching equivalent to a subdigraph of a digraph
in Theorem~\ref{thm:main} \ref{i:t1}--\ref{i:sp1}.

Conversely, the digraphs in
Theorem~\ref{thm:main} \ref{i:t1}--\ref{i:sp1} have
spectral radius at most $2$ by Theorem~\ref{thm:Gary}.
The digraphs in Theorem~\ref{thm:main} \ref{i:t1}--\ref{i:t2} contain
a subdigraph which is switching equivalent to $C_4$, so they have
spectral radius exactly $2$. It can be checked directly that
the digraphs in Theorem~\ref{thm:main} \ref{i:sp1}
have spectral radius exactly $2$.
This completes the proof of the
first part of Theorem~\ref{thm:main}.

\section{Spectral radius less than $2$}\label{sec:4}

Throughout this section, we let
$\Delta$ be a connected digraph with spectral radius less than $2$,
and complete the proof of the second part of Theorem~\ref{thm:main}.
Lemma~\ref{lem:SDspec} shows that $\cS(\Delta)$ has spectral radius less than $2$.
Connected signed graphs with  spectral radius less than $2$
are essentially classified by the following theorem
%Lemma~\ref{lem:asg2} shows that the underlying graph of $\cS(\Delta)$
%cannot be the underlying graph of $\square_{a,0,c,0}$ for $a,c>0$, since
%there are exactly two vertices of degree $3$, and they have common neighbors.
%Therefore, by the classification
due to McKee and Smyth.
%the associated signed graph of a connected digraph with spectral
%radius $<2$ must be, apart from finite exceptional cases,
%$O_{2k}$, $2O_{2k}$, or $2Q_{hk}$. In the
%last case, the underlying graph of the digraph is $Q_{hk}$,
%so it must be four-way switching equivalent to $\square_{a,0,c,0}$
%by \cite{GM2}. As for exceptional cases, only $U_1$ and $U_6$ admit
%a partition with properties described in Lemma~\ref{lem:asg2}.
%As for digraphs $\Delta$ with $\cS(\Delta)\cong O_{2k}$, we certainly have
%$C_k$ (in the notation of \cite{GM2}). There are digraphs
%$D_k$, $\tilde{C}_k$, $\tilde{C}'_k$, $\tilde{C}''_k$.
%The main tool is the following:

%\cite[Theorem 4]{MS}: ``Uncharged, singed, $(-2,2)$'', on page 282.
\begin{theorem}[{\cite[Theorem~4]{MS}}]\label{thm:McS}
Up to equivalence, the  connected signed  graphs maximal with
respect to having all their eigenvalues in $(-2,2)$ are the eleven
$8$-vertex sporadic examples $U_1,\dots,U_{11}$ shown in Fig.~\ref{fig:F-U},
and the infinite family $ O_{2k}$ of $2k$-cycles with
one negative edge
%edge of sign $-1$,
for $2k\ge 8$, shown in Fig.~\ref{fig:QO}.
Further, every  connected
%cyclotomic
signed graph having all its
eigenvalues in $(-2,2)$ is either contained in a maximal one, or
is a subgraph of one of the signed graphs $Q_{hk}$ of Fig.~\ref{fig:QO}
for $h+k\ge 4$.
\end{theorem}

In the notation of \cite{GM2}, the digraphs $C_3$ and $D_3$ have isomorphic
associated signed graph, which is $O_6$ in the notation of \cite{MS}. This means
that there are digraphs which are not switching equivalent, but their
associated signed graphs are switching equivalent. Thus, the results of
\cite{MS} does not immediately imply those of \cite{GM2}. In this section,
we show how to derive the results of \cite{GM2} from \cite{MS}.

\begin{lemma}\label{lem:O2k}
If $\cS(\Delta)$ has a connected component which is equivalent to a subgraph
of $O_{2k}$, then $\Delta$ is switching equivalent to one of
the digraphs in Theorem~\textup{\ref{thm:main}}
\ref{i:Dn}--\ref{i:path}.
\end{lemma}
\begin{proof}
Suppose first
$\cS(\Delta)$ has a component which is equivalent to
$O_{2k}$. Then $G(\cS(\Delta))$ is $C_{2k}$ or $2C_{2k}$ by Lemma~\ref{lem:1}.
In the former case, Lemma~\ref{lem:asg2} implies that $G(\Delta)$ is $C_k$,
while in the latter case, Lemma~\ref{lem:1} implies that $G(\Delta)$ is $C_{2k}$.
Then the classification follows from
\cite[Theorem~3.9]{GM2} (see also
\cite[Proposition~8.8]{GM}).

If $\cS(\Delta)$ has a component which is equivalent to
a proper subgraph of $O_{2k}$, then such a component is a path.
If $\cS(\Delta)$ is connected, then by
By Lemma~\ref{lem:asg2}~(iii), the sum of the degrees of vertices
of $G(\Delta)$ is odd, which is impossible. Thus
$\cS(\Delta)$ consists of two paths. Then by Lemma~\ref{lem:1},
$H(\Delta)$ is strongly equivalent to the adjacency matrix of
a path. By Lemma~\ref{lem:4way},
$\Delta$ is switching equivalent
to a path.
\end{proof}

\begin{lemma}\label{lem:Qhk}
If $\cS(\Delta)$ has a connected component which is equivalent to a subgraph
of $Q_{hk}$, then $\Delta$ is switching equivalent to one of the
digraphs in Theorem~\textup{\ref{thm:main}}
\ref{i:path}--\ref{i:Y}.
%\textup{(viii)--(x)}.
%following digraphs:
%\begin{enumerate}
%\item $\square_{a,0,c,0}$ for some $a,c\geq0$,
%\item $Y_{a,1,1}$ for some $a\geq1$,
%\item a path.
%\end{enumerate}
\end{lemma}
\begin{proof}
Lemma~\ref{lem:asg2} shows that $\cS(\Delta)$ is disconnected.
Indeed, for example, for $Q_{hk}$ with $h,k>0$,
there are exactly two vertices of degree $3$, and they have common neighbors.
Thus, by Lemma~\ref{lem:1}, $G(\Delta)$ is a subgraph is $G(Q_{hk})$.
Suppose $\Delta$ contains a quadrangle $Q$.
Then by Lemma~\ref{lem:O2k}, $Q$ is switching equivalent to
$\tilde{C}_4$, $\tilde{C}'_4$ or $\tilde{C}''_4$.
In the latter two cases, by
Lemma~\ref{lem:asg2}, $\cS(Q)$ is connected, and hence an octagon,
a contradiction. Thus $Q$ is switching equivalent to
$\tilde{C}_4$.
This means that $\Delta$ is switching equivalent to $\square_{a,0,c,0}$
for some $a,c\geq0$.

Now we can assume that $G(\Delta)$ is a proper subgraph of $G(Q_{hk})$
not containing a quadrangle. Then $G(\Delta)$ is $Y_{a,1,1}$ or a path,
and hence $\Delta$ is switching equivalent to $Y_{a,1,1}$ or a path.
\end{proof}

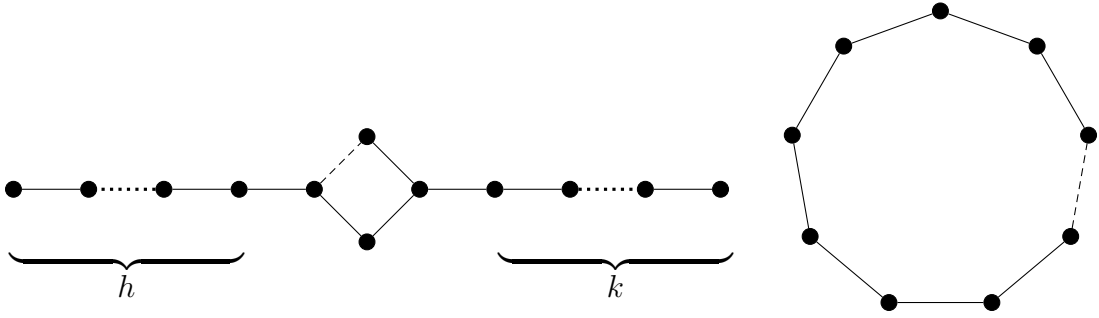
\begin{figure}
\begin{center}
\begin{tikzpicture}[auto, scale=1]
\begin{scope}
\foreach \type/\pos/\name in {{vertex/(0,0)/a0}, {vertex/(1,0)/a1}, {vertex/(2,0)/a2}, {vertex/(3,0)/a3},
{vertex/(4,0)/b0}, {vertex/(4.7,-0.7)/bm}, {vertex/(4.7,0.7)/bp}, {vertex/(5.4,0)/b1},
{vertex/(6.4,0)/c3}, {vertex/(7.4,0)/c2}, {vertex/(8.4,0)/c1}, {vertex/(9.4,0)/c0}}
 \node[\type] (\name) at \pos {};
\foreach \edgetype/\source/\dest in
 {pedge/a0/a1, dedge/a2/a1, pedge/a2/a3, pedge/a3/b0,
  pedge/b0/bm, nedge/b0/bp, pedge/bm/b1, pedge/bp/b1,
  pedge/b1/c3, pedge/c2/c3, dedge/c2/c1, pedge/c0/c1}
 \path[\edgetype] (\source) -- node[weight] {} (\dest);
\foreach \pos/\name in {{(1.5,-1)/\underbrace{\phantom{MMMMMMM}}_{\mbox{$h$}}},
 {(8,-1)/\underbrace{\phantom{MMMMMMM}}_{\mbox{$k$}}}}
 \node at \pos {$\name$};
\end{scope}
\end{tikzpicture}
\quad
%\caption{$Q_{hk}$}
%\end{center}
%\label{fig:Qhk}
%\end{figure}
%\begin{figure}
%\begin{center}
\begin{tikzpicture}[auto, scale=1] % O_{2k}
\begin{scope}
\foreach \x/\name in {{10/a1},{50/a5},{90/a9},{130/a13},{170/a17},{210/a21},{250/a25},{290/a29},{330/a33}}
 \node[vertex] (\name) at (\x:2) {};
\foreach \edgetype/\source/\dest in
 {pedge/a1/a5,pedge/a5/a9,pedge/a9/a13,pedge/a13/a17,pedge/a17/a21,pedge/a21/a25,
  pedge/a25/a29,pedge/a29/a33,nedge/a33/a1}
 \path[\edgetype] (\source) -- node[weight] {} (\dest);
\end{scope}
\end{tikzpicture}
\caption{$Q_{hk}$ and $O_{2k}$}
\label{fig:QO}
\end{center}
\end{figure}

%\begin{figure}
%\begin{center}
%\begin{tikzpicture}[auto, scale=1]
%\begin{scope}
%\foreach \type/\pos/\name in {{vertex/(0,0)/a0}, {vertex/(1,0)/a1}, {vertex/(2,0)/a2}, {vertex/(3,0)/a3},
%{vertex/(4,0)/b0}, {vertex/(4.7,-0.7)/bm}, {vertex/(4.7,0.7)/bp}, {vertex/(5.4,0)/b1},
%{vertex/(6.4,0)/c3}, {vertex/(7.4,0)/c2}, {vertex/(8.4,0)/c1}, {vertex/(9.4,0)/c0}}
% \node[\type] (\name) at \pos {};
%\foreach \edgetype/\source/\dest in
% {pedge/a0/a1, dedge/a2/a1, pedge/a2/a3, pedge/a3/b0,
%  pedge/b0/bm, nedge/b0/bp, pedge/bm/b1, pedge/bp/b1,
%  pedge/b1/c3, pedge/c2/c3, dedge/c2/c1, pedge/c0/c1}
% \path[\edgetype] (\source) -- node[weight] {} (\dest);
%\foreach \pos/\name in {{(1.5,-1)/\underbrace{\phantom{MMMMMMM}}_{\mbox{$h$}}},
% {(8,-1)/\underbrace{\phantom{MMMMMMM}}_{\mbox{$k$}}}}
% \node at \pos {$\name$};
%\end{scope}
%\end{tikzpicture}
%\caption{$Q_{hk}$}
%\end{center}
%\label{fig:Qhk}
%\end{figure}

%\begin{figure}
%\begin{center}
%\begin{tikzpicture}[auto, scale=1] % O_{2k}
%\begin{scope}
%\foreach \x/\name in {{10/a1},{50/a5},{90/a9},{130/a13},{170/a17},{210/a21},{250/a25},{290/a29},{330/a33}}
% \node[vertex] (\name) at (\x:2) {};
%\foreach \edgetype/\source/\dest in
% {pedge/a1/a5,pedge/a5/a9,pedge/a9/a13,pedge/a13/a17,pedge/a17/a21,pedge/a21/a25,
%  pedge/a25/a29,pedge/a29/a33,nedge/a33/a1}
% \path[\edgetype] (\source) -- node[weight] {} (\dest);
%\end{scope}
%\end{tikzpicture}
%\caption{$O_{2k}$}
%\label{fig:F-O}
%\end{center}
%\end{figure}

\begin{figure}
\begin{center}
% U_1
\begin{tikzpicture}[auto, scale=0.7]
 \foreach \pos/\name in
  {(45:1)/a1,(45:2.5)/b1,
   (135:1)/a2, (135:2.5)/b2,
   (225:1)/a3, (225:2.5)/b3,
   (315:1)/a4, (315:2.5)/b4}
  \node[vertex] (\name) at \pos {};
 \foreach \edgetype/\source/\dest in
  {pedge/a1/b1, pedge/a2/b2, nedge/a3/b3, pedge/a4/b4,
   pedge/a1/a2, pedge/a2/a3, pedge/a3/a4, nedge/a4/a1,
   nedge/b1/b2, pedge/b2/b3, pedge/b3/b4, pedge/b4/b1}
 \path[\edgetype] (\source) -- (\dest);
 \foreach \pos/\name in {(0,-2.2)/U_{1}}
  \node at \pos {$\name$};
\end{tikzpicture}
\quad
% U_2
\begin{tikzpicture}[auto, scale=1]
 \foreach \pos/\name in
  {(0,0)/a0,(1,0)/a1,(2,0)/a2,(3,0)/a3,(4,0)/a4,
   (0,1)/b0,(1,1)/b1,(2,1)/b2}
  \node[vertex] (\name) at \pos {};
 \foreach \edgetype/\source/\dest in
  {pedge/a0/a1, pedge/a1/a2, pedge/a2/a3, pedge/a3/a4,
   pedge/b0/b1, pedge/b1/b2,
   pedge/a1/b1, nedge/a2/b2}
 \path[\edgetype] (\source) -- (\dest);
 \foreach \pos/\name in {(2,-1)/U_{2}}
  \node at \pos {$\name$};
\end{tikzpicture}
\quad
% U_3
\begin{tikzpicture}[auto, scale=1]
 \foreach \pos/\name in
  {(0,0)/a0,(1,0)/a1,(2,0)/a2,(3,0)/a3,(4,0)/a4,(5,0)/a5,
   (1,1)/b1,(2,1)/b2}
  \node[vertex] (\name) at \pos {};
 \foreach \edgetype/\source/\dest in
  {pedge/a0/a1, pedge/a1/a2, pedge/a2/a3, pedge/a3/a4, pedge/a4/a5,
   pedge/b1/b2,
   pedge/a1/b1, nedge/a2/b2}
 \path[\edgetype] (\source) -- (\dest);
 \foreach \pos/\name in {(2.5,-1)/U_{3}}
  \node at \pos {$\name$};
\end{tikzpicture}
\end{center}

\begin{center}
% U_4
\begin{tikzpicture}[auto, scale=1]
 \foreach \pos/\name in
  {(0,0)/a0,(1,0)/a1,(2,0)/a2,(3,0)/a3,(4,0)/a4,
   (0,1)/b0,(1,1)/b1,(2,1)/b2}
  \node[vertex] (\name) at \pos {};
 \foreach \edgetype/\source/\dest in
  {pedge/a0/a1, pedge/a1/a2, pedge/a2/a3, pedge/a3/a4,
   pedge/b0/b1, pedge/b1/b2,
   pedge/a0/b0, nedge/a1/b1, pedge/a2/b2}
 \path[\edgetype] (\source) -- (\dest);
 \foreach \pos/\name in {{(2,-0.5)/U_4}}
  \node at \pos {$\name$};
\end{tikzpicture}
\quad
% U_5
\begin{tikzpicture}[auto, scale=1]
 \foreach \pos/\name in
  {(0,0)/a0,(1,0)/a1,(2,0)/a2,(3,0)/a3,(4,0)/a4,(5,0)/a5,(6,0)/a6,
   (2,1)/b2}
  \node[vertex] (\name) at \pos {};
 \foreach \edgetype/\source/\dest in
  {pedge/a0/a1, pedge/a1/a2, pedge/a2/a3, pedge/a3/a4, pedge/a4/a5, pedge/a5/a6,
   pedge/a2/b2}
 \path[\edgetype] (\source) -- (\dest);
 \foreach \pos/\name in {{(3,-0.5)/U_5}}
  \node at \pos {$\name$};
\end{tikzpicture}
\end{center}

\begin{center}
% U_6
\begin{tikzpicture}[auto, scale=1]
%\begin{tikzpicture}[yshift=4cm, xshift=0.8cm]
 \foreach \pos/\name in {
  {(1,0)/0}, {(0.5,0.86)/60}, {(-0.5,0.86)/120},
  {(-1,0)/180}, {(-0.5,-0.86)/240},{(0.5,-0.86)/300},
  {(2,0)/x+}, {(-2,0)/x-}}
  \node[vertex] (\name) at \pos {};
 \foreach \edgetype/\source/\dest in {
  pedge/0/60, pedge/60/120, nedge/120/180, pedge/180/240,
  pedge/240/300, pedge/300/0,
  pedge/0/x+, pedge/180/x-}
  \path[\edgetype] (\source) -- (\dest);
 \foreach \pos/\name in {{(0,-1.5)/U_6}}
  \node at \pos {$\name$};
\end{tikzpicture}
\quad
% U_7
\begin{tikzpicture}[auto, scale=1]
 \foreach \type/\pos/\name in
		{{vertex/(0,1)/90},{vertex/(0,-1)/270},
		{vertex/(0.86,0.5)/30}, {vertex/(0.86,-0.5)/-30},
		{vertex/(-0.86,0.5)/150}, {vertex/(-0.86,-0.5)/210},
		{vertex/(-1.86,0.5)/y+}, {vertex/(-1.86,-0.5)/y-}}
			\node[\type] (\name) at \pos {};
 \foreach \edgetype/\source/ \dest in
		{pedge/-30/30, pedge/30/90, pedge/90/150, nedge/150/210,
		pedge/210/270, pedge/270/-30, pedge/150/y+, pedge/y+/y-, pedge/y-/210}
			\path[\edgetype] (\source) -- (\dest);
\foreach \pos/\name in {{(-0.5,-1.5)/U_7}}
%\foreach \pos/\name in {{(-0.86,0.9)/i}, {(-0.5,-1.5)/U_7}}
 \node at \pos {$\name$};
\end{tikzpicture}
\quad
% U_8
\begin{tikzpicture}[auto, scale=1]
 \foreach \pos/\name in {
  (-1,0)/am1, (0,0)/a0, (1,0)/a1, (2,0)/a2,
  (0.5,0.86)/b0, (1.5,0.86)/b1,
  (0.5,-0.86)/c0, (1.5,-0.86)/c1}
  \node[vertex] (\name) at \pos {};
 \foreach \edgetype/\source/\dest in {
  pedge/am1/a0, nedge/a0/b0, pedge/a0/c0, pedge/a1/b0, pedge/a1/c0,
  nedge/a1/a2, pedge/a2/b1, pedge/a2/c1, pedge/b0/b1, pedge/c0/c1}
  \path[\edgetype] (\source) -- (\dest);
 \foreach \pos/\name in {(0.5,-1.5)/U_8}
  \node at \pos {$\name$};
\end{tikzpicture}
\end{center}

\begin{center}
% U_9
\begin{tikzpicture}[auto, scale=1]
 \foreach \pos/\name in {
  (0,0)/a0, (1,0)/b0, (0,1)/d0, (1,1)/c0,
  (-0.7,-0.7)/a1, (1+0.7,0-0.7)/b1, (0-0.7,1+0.7)/d1, (1+0.7,1+0.7)/c1}
  \node[vertex] (\name) at \pos {};
 \foreach \edgetype/\source/\dest in {
  pedge/a0/b0, nedge/b0/c0, pedge/c0/d0, pedge/d0/a0,
  pedge/a0/a1, pedge/b0/b1, pedge/c0/c1, pedge/d0/d1}
  \path[\edgetype] (\source) -- (\dest);
 \foreach \pos/\name in {(0.5,-1)/U_9}
  \node at \pos {$\name$};
\end{tikzpicture}
\quad
% U_{10}
\begin{tikzpicture}[auto, scale=1]
 \foreach \pos/\name in {
  (0,0)/a0, (1,0)/a1, (2,0)/a2, (3,0)/a3,
  (0,1)/b0, (1,1)/b1, (2,1)/b2, (3,1)/b3}
  \node[vertex] (\name) at \pos {};
 \foreach \edgetype/\source/\dest in {
  pedge/a0/a1, pedge/a1/a2, pedge/a2/a3,
  pedge/b0/b1, nedge/b1/b2, pedge/b2/b3,
  nedge/a0/b0, pedge/a1/b1, pedge/a2/b2, nedge/a3/b3}
  \path[\edgetype] (\source) -- (\dest);
 \foreach \pos/\name in {(1.5,-1)/U_{10}}
  \node at \pos {$\name$};
\end{tikzpicture}
\quad
% U_{11}
\begin{tikzpicture}[auto, scale=1]
 \foreach \pos/\name in {
  (0,0)/a0, (1,0)/a1, (2,0)/a2, (3,0)/a3,
  (1,1)/b1, (2,1)/b2, (3,1)/b3, (4,1)/b4}
  \node[vertex] (\name) at \pos {};
 \foreach \pos/\name in {(2,-1)/U_{11}}
  \node at \pos {$\name$};
 \foreach \edgetype/\source/\dest in {
  pedge/a0/a1, pedge/a1/a2, pedge/a2/a3,
  pedge/b1/b2, pedge/b2/b3, pedge/b3/b4,
  pedge/a1/b1, nedge/a2/b2, pedge/a3/b3}
  \path[\edgetype] (\source) -- (\dest);
\end{tikzpicture}
\end{center}
\caption{$U_1,\dots,U_{11}$ as in \cite[Fig.~12]{MS}}
\label{fig:F-U}
\end{figure}
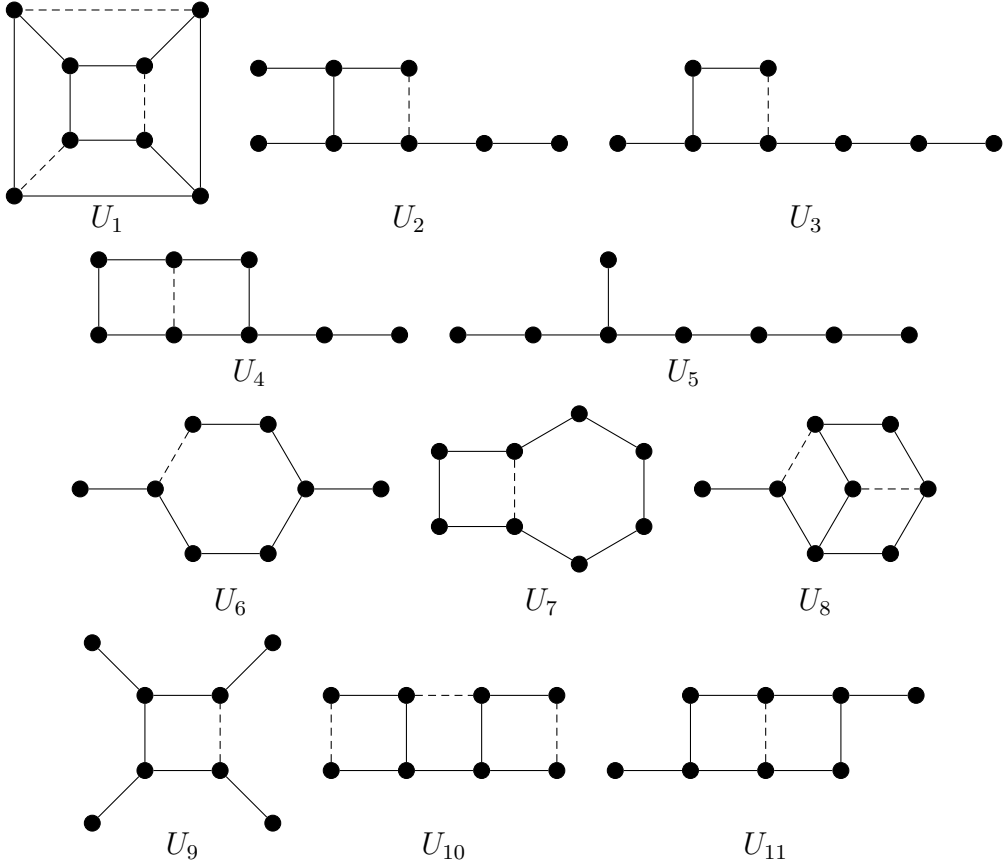

%\begin{tikzpicture}[auto, scale=1]
%	\foreach \type/\pos/\name in
%		{{vertex/(0,1)/90},{vertex/(0,-1)/270},
%		{vertex/(0.86,0.5)/30}, {vertex/(0.86,-0.5)/-30},
%		{vertex/(-0.86,0.5)/150}, {vertex/(-0.86,-0.5)/210},
%		{vertex/(-1.86,0.5)/y+}, {vertex/(-1.86,-0.5)/y-}}
%			\node[\type] (\name) at \pos {};
%	\foreach \edgetype/\source/ \dest in
%		{pedge/-30/30, pedge/30/90, wedge/150/90, wedge/210/150,
%		pedge/210/270, pedge/270/-30, wedge/150/y+, pedge/y-/y+, pedge/y-/210}
%			\path[\edgetype] (\source) -- (\dest);
%\end{tikzpicture}

\begin{lemma}\label{lem:UC}
If $\cS(\Delta)$ is connected and equivalent to a subgraph of $U_i$
for some $i$, then $i=1$ or $6$, and $\Delta$ is switching equivalent to
a path of length at most $3$,
$D_3$, $\tilde{C}_3$,
$\tilde{U}_1$ or $\tilde{U}_6$.
%one of the following digraphs:
%\begin{enumerate}
%\item $Y_1$,
%\item the digraph obtained from the directed triangle $D_3$
%by adding a vertex and an arc from this vertex to one of the vertices of $D_3$.
%\end{enumerate}
\end{lemma}
\begin{proof}
First, if $\cS(\Delta)$ is connected and equivalent to a proper subgraph of $U_i$,
then $\Delta$ has at most three vertices. Then it is routine to check that $\Delta$ is
switching equivalent to a path of length at most $3$,
$D_3$ or $\tilde{C}_3$.

Now assume that $\cS(\Delta)$ is equivalent to $U_i$.
Then the underlying
graph of $U_i$ must satisfy the conditions of Lemma~\ref{lem:asg2}.
The only graphs satisfying the conditions are $U_1,U_6$.

If $\cS(\Delta)$ is equivalent to $U_1$, then the underlying graph
of $\Delta$ is $K_4$. Then it is routine to check that $\Delta$ is
switching equivalent to $\tilde{U}_1$.

If $\cS(\Delta)$ is equivalent to $U_6$, then the underlying graph
of $\Delta$ is the triangle with a pendant edge attached.
Then it is routine to check that $\Delta$ is
switching equivalent to $\tilde{U}_6$.
\end{proof}

\begin{figure}
\begin{center}
\begin{tikzpicture}[auto, scale=1]
 \foreach \pos/\name in {
  (0:1)/a0,(120:1)/a1,(240:1)/a2,(2.5,0)/b}
  \node[vertex] (\name) at \pos {};
 \foreach \edgetype/\source/\dest in {
   wedge/a0/a1, wedge/a1/a2, wedge/a2/a0,
   pedge/a0/b, pedge/a1/b, pedge/a2/b}
 \path[\edgetype] (\source) -- (\dest);
\end{tikzpicture}
\quad
\begin{tikzpicture}[auto, scale=1]
 \foreach \pos/\name in {
  (0:1)/a0,(120:1)/a1,(240:1)/a2,(2.5,0)/b}
  \node[vertex] (\name) at \pos {};
 \foreach \edgetype/\source/\dest in {
   wedge/a0/a1, wedge/a1/a2, wedge/a2/a0,
   pedge/a0/b}
 \path[\edgetype] (\source) -- (\dest);
\end{tikzpicture}
\end{center}
\caption{$\tilde{U}_1$ and $\tilde{U}_6$}
\label{fig:U1U6}
\end{figure}
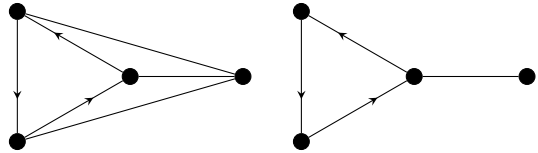

\begin{lemma}\label{lem:UDC}
If $\cS(\Delta)$ is disconnected and a connected component
of $\cS(\Delta)$ is equivalent to a subgraph of $U_i$ for some $i$, then
$\Delta$ is switching equivalent to a subdigraph of a canonical digraph of
$U_i$.
\end{lemma}
\begin{proof}
Observe that all of the signed graphs $U_1,\dots,U_{11}$ are bipartite.
The result follows from
Lemmas~\ref{lem:4way} and \ref{lem:bipartite}.
\end{proof}

Now we are ready to complete the proof of the second part of
Theorem~\ref{thm:main}.
By Lemma~\ref{lem:SDspec}, the
associated signed graph $\cS(\Delta)$ has spectral radius less than $2$.
By Theorem~\ref{thm:McS}, up to equivalence, each connected component of
$\cS(\Delta)$ is either contained in a maximal
one $O_{2k}$, $U_1,\dots,U_{11}$, or is a subgraph of
$Q_{hk}$ for some $h,k$ with  $h+k\geq4$.
If $\cS(\Delta)$ has a connected component which is equivalent to a subgraph
of $O_{2k}$
(resp.\ $Q_{hk}$),
then Lemma~\ref{lem:O2k}
(resp.\ Lemma~\ref{lem:Qhk})
shows that $\Delta$ is switching equivalent to one of the digraphs
in Theorem~\ref{thm:main}
\ref{i:Dn}--\ref{i:path} (resp.\ \ref{i:path}--\ref{i:Y}).
%(iv)--(viii) (resp.\ (viii)--(x)).
%the conclusion of Theorem~\ref{thm:main} holds.
%If $\cS(\Delta)$ has a connected component which is equivalent to a subgraph
%of $Q_{hk}$, then Lemma~\ref{lem:Qhk} shows that the conclusion
%of Theorem~\ref{thm:main} holds.
If $\cS(\Delta)$ is connected and equivalent to a subgraph
of $U_i$,
then Lemma~\ref{lem:UC} shows that
$\Delta$ is switching equivalent to one of the digraphs
in Theorem~\ref{thm:main} \ref{i:Dn}, \ref{i:Cn}, \ref{i:path},
\ref{i:U1} or \ref{i:U6}.
%(iv), (v), (viii), (xi) or (xii).
If $\cS(\Delta)$ is disconnected and
has a connected component which is equivalent to a subgraph
of $U_i$, then Lemma~\ref{lem:UDC} shows that
$\Delta$ is switching equivalent to a subdigraph of one of the digraphs
in Theorem~\ref{thm:main} \ref{i:canonical}.
%(xiii).

Conversely,  every digraph in \ref{i:Dn}--\ref{i:canonical}
%\textup{(iv)}--\textup{(xiii)}
has spectral radius less than $2$ by
Lemma~\ref{lem:SDspec} and Theorem~\ref{thm:McS}.
This completes the proof of Theorem~\ref{thm:main}.

Theorem~\ref{thm:main} and \cite[Theorem~4.15]{GM2}
give the same infinite families of digraphs
$D_n$, $\tilde{C}_n$, $\tilde{C}'_n$,  $\tilde{C}''_n$, $P_n$,
$\square_{a,0,c,0}$ and
$Y_{a,1,1}$. Since Theorem~\ref{thm:main} only claims that
every digraph with spectral radius less than $2$ is a subdigraph of
a digraph listed there, those which are listed in
\cite[Theorem~4.15]{GM2} but are not maximal do not appear
in Theorem~\ref{thm:main}.
Table~\ref{tab:1} gives the list of graphs apart from the infinite families
in Theorem~\ref{thm:main} and \cite[Theorem~4.15]{GM2}.
The first column gives the item labels in \cite[Theorem~4.15]{GM2},
and the second column gives the name of the digraphs in each item,
except item (g) for which the digraph has  no name.
The third column gives the names of digraphs in our notation,
where $\Delta(U)$ denotes a canonical digraph of a bipartite signed
graph $U$.
The last column will be explained in Section~\ref{sec:cr}.
The digraph $\Delta(U_7)$ is missing in \cite{GM2},
giving a counterexample
to the statement of \cite[Lemma~4.8(b)]{GM2}.

\begin{table}[h]
\begin{center}
\begin{tabular}{|c|c|c|c|}
\hline
\multicolumn{2}{|c|}{Notation in \cite{GM2}}& Theorem~\ref{thm:main}
&lattice\\ \hline
(f)& $Y_{4,2,1}$ & $\Delta(U_5)$ & $E_8\otimes\Z[i]$\\
& $Y_{3,2,1}$ & $\subset Y_{4,2,1}$ & $E_7\otimes\Z[i]$ \\ \hline
(g) & --- & $\tilde{U}_6$ & $E_8^{\C}$\\ \hline
(h) & $Y_1$ & $\tilde{U}_1$& $E_8^{\C}$\\ \hline
(j) & $\square_{3,1,0,0}$ & $\Delta(U_3)$& $E_8\otimes\Z[i]$\\
    & $\square_{2,1,1,0}$ & $\Delta(U_2)$& $E_8\otimes\Z[i]$\\
    & $\square_{1,1,1,1}$ & $\Delta(U_9)$& $E_8\otimes\Z[i]$\\ \hline
(k) & $X_1$ & $\subset X_2$ & $E_6\otimes\Z[i]$ \\
    & $X_2$ & $\subset X_3$ & $E_7\otimes\Z[i]$ \\
    & $X_3$ & $\Delta(U_{11})$ & $E_8\otimes\Z[i]$\\
    & $X_4$ & $\Delta(U_4)$ & $E_8\otimes\Z[i]$\\
    & $X_5$ & $\subset X_6$ & $E_7\otimes\Z[i]$ \\
    %$\Delta(U_1-\text{vertex})$\\
    & $X_6$ & $\Delta(U_8)$ & $E_8\otimes\Z[i]$\\
    & $X_7$ & $\Delta(U_{10})$ & $E_8\otimes\Z[i]$\\
    & $X_8$ & $\Delta(U_1)$ & $E_8\otimes\Z[i]$\\ \hline
(l) & $X_9$ & $\subset X_{10}$ & $E_7\otimes\Z[i]$\\
    & $X_{10}$ & $\Delta(U_6)$ & $E_8\otimes\Z[i]$\\ \hline
    &  & $\Delta(U_7)$ & $E_8\otimes\Z[i]$\\ \hline
\end{tabular}
\end{center}
\caption{Comparison of Theorem~\ref{thm:main} and \cite[Theorem~4.15]{GM2}}
\label{tab:1}
\end{table}

\section{Concluding remarks}\label{sec:cr}

By a $\Z$-\textbf{lattice} (or simply, a \textbf{lattice}),
we mean a free $\Z$-module equipped with a positive definite
symmetric bilinear form.
A $\Z[i]$-\textbf{lattice}, which we call a
\textbf{Gaussian lattice}, can be defined in an
analogous manner using a Hermitian form instead of
a symmetric bilinear form.
Since the real part of a positive definite Hermitian form is
a positive definite symmetric bilinear form,
a Gaussian lattice $\Lambda$
is also a $\Z$-lattice by regarding $\Lambda$ as a
$\Z$-module
%equipped with the quadratic form defined by this positive
%definite inner product.
Conversely, if $L$ is a lattice, then one can equip a
positive definite Hermitian form on $L\otimes\Z[i]$,
making it a Gaussian lattice. Not every Gaussian lattice
is obtained in this way.

A (Gaussian) lattice is called a (Gaussian)
\textbf{root lattice} if it is generated by its set of vectors of
squared norm $2$.
A (Gaussian) lattice is said to be \textbf{irreducible} if it is not
an orthogonal direct sum of proper sublattices.
Note that a (Gaussian) root lattice is irreducible if and only if
the non-orthogonality graph on the set of its roots is connected.
Using the classification of irreducible root lattices \cite{Ebeling},
every irreducible Gaussian root lattice $\Lambda$
is either $L\otimes\Z[i]$ for
some irreducible root lattice $L$, or $\Lambda$
is irreducible as a $\Z$-lattice.
The latter possibilities are classified by \cite[Lemma~3.1]{KM}:
\begin{enumerate}
\item\label{irred1} The root lattice of type $\mathsf{D}_{2n}$
may be regarded as a Gaussian lattice of rank $n$
\[\mathsf{D}_{2n}^{\C}=
\{\sum_{j=1}^n a_je_j\mid a_j\in\Z[i],\;\sum_{j=1}^n a_j\in
(1+i)\Z[i]\},\]
where $e_1,\dots,e_n$ are the standard orthonormal basis.
\item\label{irred2} The root lattice of type $\mathsf{E}_8$
may be regarded as a
Gaussian lattice $\mathsf{E}_8^{\C}$ of rank $4$
(see \cite[p.~373]{I}).
\end{enumerate}

Every connected
digraph with spectral radius at most $2$ gives rise to an
irreducible Gaussian
root lattice. In fact, if $H$ is the Hermitian adjacency matrix of
a digraph and $H$ has spectral radius at most $2$, then $2I-H$
is positive semidefinite. This implies that there is a Gaussian lattice
$\Lambda$ generated by a set $X$ of vectors
of squared norm $2$ in $\Lambda$ such that
$-H$ is the displaced Gram matrix of $X$.
Thus, the lattice $\Lambda$ is a Gaussian root lattice.
In particular,
every maximal digraph with  spectral radius $2$ in Theorem~\ref{thm:main}
generates one of the Gaussian
lattices \ref{irred1}, \ref{irred2} or
$L\otimes\Z[i]$ for some irreducible root lattice $L$.
%$A_n\otimes\Z[i]$, $D_n\otimes\Z[i]$, $E_n\otimes\Z[i]$,
%$D_{2n}$ in (i), $E_8$ in (ii) above.
Also, every digraph with  spectral radius less than $2$
gives a basis of a Gaussian root lattice.
%In particular, graphs classified by \cite{GM}
%form a basis of one of the above lattices.
Table~\ref{tab:root1}
(resp.\ the last column of Table~\ref{tab:1})
shows the correspondences between the
maximal digraphs with spectral radius exactly $2$
(resp.\ less than $2$) and
the Gaussian root lattices.
It is worth mentioning that the associated signed graph
$\cS(\Delta)$ of a digraph $\Delta$ with spectral radius at most $2$
is connected if and only if the corresponding Gaussian root
lattice $\Lambda$ is irreducible as a $\Z$-lattice.
Indeed, $X$ is a subset of $\Lambda$ whose displaced Gram
matrix is $-H(\Delta)$, then the real part of
the displaced Gram matrix of $X\cup iX$ is the adjacency matrix
of $\cS(\Delta)$. Thus,
$\Lambda$ is irreducible as a $\Z$-lattice if and only if
$\cS(\Delta)$ is connected.

\begin{table}[h]
\begin{tabular}{|c|c|c|}
\hline
$\Delta_{2k}^{(1)}$&$T_{2k}$&$D_k\otimes\Z[i]$\\
$\Delta_{2k}^{(i)}$&$T_{2k}^{(i)}$&$D_{2k}^{\mathbb{C}}$\\
$\Delta_{14}$&$S_{14}$&$E_7\otimes\Z[i]$\\
$\Delta_{16}$&$S_{16}$&$E_8\otimes\Z[i]$\\
$\Delta_{8}^\dagger$&$S_8^\dagger$&$E_8^{\mathbb{C}}$\\
\hline
\end{tabular}
\caption{Gaussian root lattices}
\label{tab:root1}
\end{table}

%\newpage
%\section{Connected digraphs with smallest eigenvalue greater than $-\sqrt{2}$}
%\label{sec:-sqrt2}

As is well known, a connected simple graph with smallest eigenvalue at least $-1$
is a complete graph.
In fact, there are no graphs with smallest eigenvalue in $(-\sqrt{2},-1)$.
Guo and Mohar \cite[Proposition~8.6]{GM} determined digraphs
with the same spectrum as a complete graph. It turns out that
digraphs with the same spectrum as a complete graph are switching
equivalent to a complete graph.
We strengthen \cite[Proposition~8.6]{GM} by showing
that a digraph with smallest eigenvalue greater than $-\sqrt{2}$ is
switching equivalent to a complete graph.

For a Hermitian matrix $H$,
we denote by $\lambda_{\min}(H)$ the smallest
eigenvalue of $H$, and
for a digraph $\Delta$, we write $\lambda_{\min}(\Delta)
=\lambda_{\min}(H(\Delta))$.

\begin{proposition}\label{prop:-sqrt2}
Let $\Delta$ be a connected digraph with $\lambda_{\min}(\Delta)>-\sqrt{2}$.
Then $\Delta$ is switching equivalent to a complete graph.
\end{proposition}
\begin{proof}
By Lemma~\ref{lem:SDspec}, $\lambda_{\min}(A(\cS(\Delta)))>-\sqrt{2}$.
Then each connected component of $\cS(\Delta)$ is
switching equivalent to a complete graph by \cite[Prop.~4.7]{GMSS}.
Since $\cS(\Delta)$ cannot be complete by Lemma~\ref{lem:asg2}~(i),
$\cS(\Delta)$ is disconnected. Now Lemma~\ref{lem:1} implies that
$H(\Delta)$ is strongly equivalent to the adjacency matrix
of a complete graph. The result then follows from Lemma~\ref{lem:4way}.
\end{proof}

\subsection*{Acknowledgements}
We thank Gary Greaves for letting us copy and modify the diagrams in
\cite{G}.

%The research of Alexander Gavrilyuk is partially supported by JSPS KAKENHI Grant Number
%18K03395.

%\newpage

\end{document}